\documentclass[11pt]{article}
\usepackage{epsfig}
\usepackage{epstopdf}
\usepackage{graphics,graphicx}
\usepackage{amssymb, latexsym}
\usepackage{amsmath, amsfonts}

\usepackage{multirow}
\usepackage{multicol}

\usepackage{amsthm}

\usepackage{authblk}
\usepackage{array}
\usepackage{tabu}
\usepackage{color}
\usepackage{subfigure}
\usepackage{placeins}
\usepackage{float}

\usepackage{cleveref}

\usepackage{commath} 

\usepackage{booktabs}

\usepackage[ruled,vlined]{algorithm2e}

\usepackage{tikz}
\usepackage{tikz-3dplot}

\usepackage{textcomp}
\usepackage{pgfplots}
\pgfplotsset{width=10cm,compat=1.9}

\numberwithin{equation}{section}

\newtheorem{theorem}{Theorem}[section]

\newtheorem{remark}[theorem]{Remark}
\newtheorem{lemma}[theorem]{Lemma}

\newcommand {\real} {\,\hbox{{\rm R \kern-1.2em I \kern.3em }}}

\def \bV{{\mathbf V}}

\def \bX{{\mathbf X}}

\def \bL{{\mathbf L}}
\def \bI{{\mathbf I}}
\def \bv{{\mathbf v}}
\def \b0{{\mathbf 0}}
\def \bu{{\mathbf u}}

\def \bw{{\mathbf w}}
\def \bbf{{\mathbf f}}
\def \bH{{\mathbf H}}

\def \by{{\mathbf y}}

\def \bn{{\mathbf n}}

\def \bb{{\mathbf b}}
\def \bg{{\mathbf g}}

\def \bz{{\mathbf z}}

\def \bv{{\mathbf v}}
\def \bf{{\mathbf f}}

\def \bsigma{\mbox{\boldmath $\sigma$}}

\def \bSigma{\mbox{\boldmath $\Sigma$}}
\def \bLambda{\mbox{\boldmath $\Lambda$}}

\def \beeta{\mbox{\boldmath $\eta$}}
\def \bxi{\mbox{\boldmath $\xi$}}

\def \nn{\nonumber}

\def\el {\nonumber }

\def \Omegn1{\Omega^f_{t_{n+1}}}

\def \beq{\begin{equation}}
\def \eeq{\end{equation}}
\def \beqn{\begin{eqnarray}}
\def \eeqn{\end{eqnarray}}

\title{Decoupling methods for fluid-structure interaction with local time-stepping}

\author[1]{Hemanta Kunwar
\thanks{ Partially supported by the NSF under grant number DMS-2207971. email: hkunwar@g.clemson.edu}}
\author[2]{Hyesuk Lee
  \thanks{ Partially supported by the NSF under grant number DMS-2207971. email: hklee@clemson.edu}}
\affil[1,2]{School of Mathematical and Statistical Sciences, Clemson University, Clemson, SC 29634-0975 }

\date{}
\begin{document}

\maketitle 
\begin{abstract}
 We introduce two global-in-time domain decomposition methods, namely the Steklov-Poincare method and the Robin method, for solving a fluid-structure interaction system. These methods allow us to formulate the coupled system as a space-time interface problem and apply iterative algorithms directly to the evolutionary problem. Each time-dependent subdomain problem is solved independently, which enables the use of different time discretization schemes and time step sizes in the subsystems. This leads to an efficient way of simulating time-dependent phenomena. We present numerical tests for both non-physical and physical problems, with various mesh sizes and time step sizes to demonstrate the accuracy and efficiency of the proposed methods. 

\end{abstract}

\setcounter{equation}{0}
\setcounter{figure}{0}

\section{Introduction}\label{sec: intro}
The Fluid-Structure Interaction (FSI) problems are multiphysics problems where the fluid flow and an elastic structure are coupled through the continuity of traction force and velocity on the interface. FSI systems have a wide range of applications in various fields, including manufacturing, energy, aeroelasticity, defense, and biology \cite{R1, R2, R3, R4, R5, R6, R7, R9, R8}. In engineering, FSI systems are considered in designing inkjet printers, blades for wind turbines, airplane wings, combustion chambers in engines, and offshore oil rigs. In biology, such systems are often considered to study blood flow through vessels.

The FSI system is considered as a coupled monolithic system in \cite{R11, R12, R13, mtg, mng}. In such an approach, the computational complexity arises from solving a large matrix system, necessitating the use of an efficient and suitable preconditioner for the discretized system \cite{mng}. An alternative approach involves decoupling the fluid and structure subsystems \cite{bfpt, BRM17, bdg, B, A, BF09, Fer11, FGG05, Klsee, R14, vad, R10}. Implementing such methods, despite their advantages of using partitioned solvers and smaller matrices for each subsystem, can pose challenges in achieving efficient iteration between the two subsystems.

There have been extensive studies on domain decomposition (DD) techniques for FSI in the literature. Various approaches have been considered, including explicit schemes \cite{BF09, Fer11} and semi-implicit schemes \cite{BRM17, FGG05, NBR21}. 
Many implicit DD methods have also been investigated for better stability of the numerical solution. For example, an implicit DD method based on optimization is considered in \cite{R10}, for both linear and nonlinear elastic formulations for the structure. There, the stress force on the interface is used as a Neumann control, which is updated until the stress discontinuity on the interface is sufficiently small, by enforcing the continuity of velocity through a Dirichlet boundary condition for the fluid subsystem. This process requires solving the subsystems in serial. Another optimization approach for FSI is explored in \cite{R14} by formulating the FSI problem as a least squares problem, where the jump in the velocities of the two substructures is minimized by a Neumann control enforcing the continuity of stress on the interface. In \cite{JMP}, the hybridizable discontinuous Galerkin (HDG) finite element method is used in the simulation of FSI. The coupling between an underlying incompressible fluid and an embedded solid is formulated through the overlapping domain decomposition method in conjunction with a mortar approach in \cite{hab}. A fictitious domain approach, where the fluid velocity and pressure are extended into the solid domain by introducing new unknowns, has been applied to study FSI in \cite{bfpt, bdg, vad}. In \cite{A} a splitting scheme based on Robin conditions is analyzed with an additional variable representing the structure velocity, where a common Robin parameter is utilized for both fluid and structure sub-problems. This approach uses common time steps for the fluid and the structure sub-problems, and the loosely coupled subproblems are solved at each time step. 

In \cite{B}, the finite element approximation of the DD formulation introduced in \cite{A} is analyzed, and an error estimate is derived for the fully discretized system. 
Loosely coupled schemes based on interface conditions of Robin type are also found in \cite{gig} for the time-discretized FSI system, where the choice of optimal Robin parameters is analyzed. 

In classical DD approaches for time-dependent problems,  model equations are discretized uniformly in time, and DD methods are implemented at each time step as a steady-state problem. However, using a uniform time step throughout the entire time domain can be inefficient in certain FSI applications where the time scales of the fluid and structure domains differ. In recent studies, alternative approaches based on global-in-time or space-time domain decomposition (DD) have been employed where iterative algorithms are directly applied to the evolutionary problem. This enables the independent solving of subdomain problems with the different time discretization schemes and step sizes, resulting in an efficient simulation of time-dependent phenomena.

The space-time DD approach has been extensively investigated for porous medium flows (see \cite{HJJ13, HJK16,  Yotov23} and the references therein) and recently studied for the Stokes-Darcy systems \cite{THH, THe}. In \cite{THe}, a global-in-time DD method is developed based on the physical transmission conditions for the nonlinear Stokes-Darcy coupling. A time-dependent Steklov-Poincare type operator is constructed, and non-matching time grids are implemented using $L^2$ projection functions to exchange data on the space-time interface between different time grids.  Another global-in-time DD method is proposed in \cite{THH} for the mixed formulation of the non-stationary Stokes-Darcy system based on Robin transmission conditions.

This work aims to study the global-in-time DD methods introduced in \cite{THH, THe} for an FSI system using nonconforming time discretization. We consider two different DD schemes, based on the Steklov-Poincar\'{e} operator and Robin transmission conditions, respectively.  
To our knowledge, the global-in-time DD scheme has not been considered for an FSI system in the literature. Apart from the advantage of using local time stepping, another key point of this work is that we were able to simulate hemodynamic application problem using the Steklov-Poincar\'{e} method without encountering the stability issue. In the analysis of the DD method for the FSI system, additional difficulty arises due to the model equations of hyperbolic and parabolic types, unlike the Stokes-Darcy system. Another issue is caused by the lack of regularity of the unknown Lagrange multiplier function. 

The paper is structured as follows. Section~\ref{sec: model} introduces the FSI model system. In Section~\ref{sec:SWR}, we derive space-time interface problems for the continuous FSI model, and present the Schwarz waveform relaxation (SWR) algorithm and its convergence analysis. The semi-discrete system, the discrete SWR algorithm, and its convergence analysis are discussed in Section~\ref{sec:semidiscrete}. Section~\ref{sec:numericalresults} presents the results of numerical tests conducted on two examples. Following that, Section~\ref{sec:conclusion} provides the conclusion.

\section{Model equations}\label{sec: model}
The FSI problem involves coupling an incompressible Newtonian fluid with a linear
elastic structure. To simplify the problem and conduct a rigorous analysis, we 
assume that the fluid is governed by the linear Stokes equations in a fixed domain. 
However, the proposed DD schemes can be extended to a nonlinear FSI system, as demonstrated in \cite{THe}. See Remark \ref{remark:nonlinear} for more details.

Suppose the domain under consideration comprises two bounded regions $\Omega_f, \Omega_s \in \real^d, d = 2,3$, separated by the common interface $\Gamma$. See Figure ~\ref{fig:domain}.
The free fluid occupies the first region $\Omega_f$ and has boundary $\partial \Gamma_f=\Gamma_f \cup \Gamma$. A saturated elastic structure occupies the second region $\Omega_s$ with the boundary $\partial \Gamma_s = \Gamma_s \cup \Gamma$.

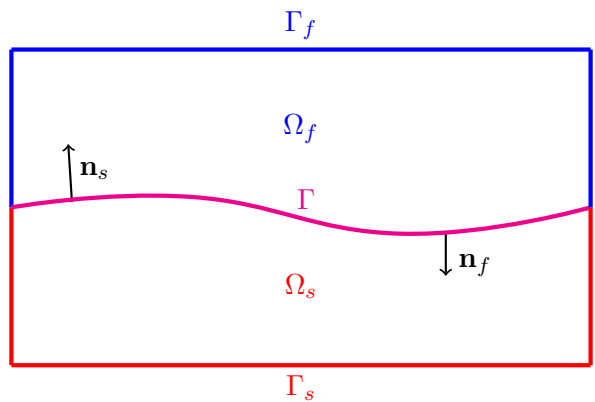
\begin{figure}[H]
\begin{center}
\begin{tikzpicture}[scale=0.7]
\draw[ultra thick, -, color=red] (0,0)--(11,0);
\draw[ultra thick, -, color=blue] (0,6)--(0,3);
\draw[ultra thick, -, color=red] (0,3)--(0,0);

\draw[thick, ->] (1.15,3.1)--(1.08,4.2);
\node[right] at (1.1,3.7) {$ \bn_s $};

\draw[thick, ->] (8.25,2.55)--(8.25,1.7);
\node[right] at (8.3,1.9) {$ \bn_f $};

\draw[ultra thick, -, color=blue] (0,6)--(11,6);
 \draw[ultra thick, -, color=magenta] plot [smooth,tension=0.8] coordinates{(0,3) (3.4,3.2) (7.4,2.5) (11,3)};
\draw[ultra thick, -, color=red] (11,0)--(11,3);
\draw[ultra thick, -, color=blue] (11,3)--(11,6);


\node[below] at (5.5,0) {$ \textcolor{red}{\Gamma_s} $};
\node[above] at (5.5,6) {$ \textcolor{blue}{\Gamma_f} $};

\node[] at (5.5,4.5) {$ \textcolor{blue}{\Omega_f} $};
\node[] at (5.5,1.5) {$ \textcolor{red}{\Omega_s} $};

\node[] at (5.6,3.15) {$ \textcolor{magenta}{\Gamma} $};
\end{tikzpicture}
\caption{Two-dimensional domain formed by FSI system}
\label{fig:domain}
\end{center} \vspace{-0.4cm}
\end{figure}
\noindent
Consider the fluid equations:
\begin{eqnarray}\label{stokesequn}
\rho_f  \partial_t \bu -2\nu_{f} \, \nabla \cdot D(\bu) + \nabla
p &=& \bbf_f  \quad \mbox{in } \Omega_f \times (0, T) \,, \label{moment}
\\
\nabla \cdot \; \bu &=& 0  \quad \mbox{in } \Omega_f \times (0, T) \,, \label{mass}\\
\bu &=& \textbf{0} \quad \mbox{in } \Gamma_f \times (0, T) \,, \label{fluidbd} \\
\bu(., 0) &=&\bu^{0}   \quad \mbox{in } \Omega_f, \label{fluidin}
\end{eqnarray}
where $\bu$ denotes the velocity vector of the fluid, $p$ the pressure of the fluid, $\rho_f$ the density of the fluid, $\nu_f$ the fluid viscosity, and $\bbf_f$ the body force acting on the fluid. 
Here, $D(\bu)$ is the strain rate tensor
\begin{equation*}
D(\bu) = \frac{1}{2}\left( \nabla \bu + (\nabla \bu)^T \right )
\end{equation*}
and the Cauchy stress tensor is given by
\begin{align}
\bsigma_f =  2\nu_{f}\, D(\bu)  - p \bI. \el
\end{align}
The equation (\ref{moment}) represents the conservation of linear momentum, while equation (\ref{mass}) represents the conservation of mass.
The elastic system is represented by:
\begin{eqnarray}\label{biotequn}
\rho_s \partial_t^2 \beeta -2 \nu_s \, \nabla \cdot D(\beeta)  - \lambda \nabla (\nabla \cdot \beeta) &=& \bbf_s \quad \mbox{in } \Omega_s \times (0, T) \,, \label{structure}
 \\
 \beeta &=& \textbf{0} \quad \mbox{in } \Gamma_s \times (0, T) \,, \label{porobd}
 \\
\beeta(.,0) &=& \beeta^{0}  \quad \mbox{in } \Omega_s \,,\label{poroin1}
\\
\partial_t\beeta(.,0) &=& \bar{\beeta}^{0} \quad \mbox{in } \Omega_s \,,
\label{poroin}
\end{eqnarray}
where $\beeta$ is the displacement of the structure and $\bbf_s$ is the body force.
The total stress tensor for the elastic structure is given by
\begin{align*}
\bsigma_s = 2 \nu_s D(\beeta) + \lambda (\nabla \cdot \beeta )\bI,
\end{align*}
where $\nu_s$ and $\lambda$ denote the Lam\'e constants. 
The density of the elastic structure is denoted by $\rho_s$. 

The fluid and elastic models, (\ref{moment})-(\ref{fluidin}) and (\ref{structure})-(\ref{poroin}), are coupled via the following interface conditions:
\begin{eqnarray}
\bsigma_f  \bn_f & =  &- \bsigma_s  \bn_s    \quad \mbox{on } \Gamma \times (0,T)
 \,, \label{inter1} \\
 \partial_t \beeta &=& \bu \quad \mbox{on } \Gamma \times (0,T)
 \,, \label{inter2}
\end{eqnarray}
where $\bn_f$ and $\bn_s$ denote outward unit normal vectors to $\Omega_f$ and $\Omega_s$, respectively.
 These interface conditions suffice to precisely couple the Stokes system \eqref{moment}-\eqref{fluidin} to the structure system (\ref{structure})-(\ref{poroin}), imposing the balance of normal stresses and  the continuity of velocity. 

To establish a weak formulation of the problem, we adopt standard notation for Sobolev spaces and their associated norms and seminorms. For $S \subset \real^d$, the norm for the Hilbert space $H^m(S)$ is denoted by  $\|\cdot \|_{m,S}$. For $m=0$, $( \cdot, \cdot)_{S}$ and $\|\cdot \|_{S}$ denote the inner product and the norm in $L^2(S)$ respectively. Moreover, if $S = \Omega_f$ or $\Omega_s$, and the context is clear, $S$ will be omitted, i.e., $(\cdot,\cdot) = (\cdot,\cdot)_{\Omega_f}$
or $(\cdot,\cdot)_{\Omega_s}$ for functions defined in $\Omega_f$ and
$\Omega_s$. For $F \subset \real^{d-1}$ such that $F \subset \partial \Omega_{f} \cap \partial
\Omega_{s}$, we use $\langle \cdot , \cdot \rangle_{F}$ to denote
the duality pairing between $H^{-1/2}(F)$ and $H^{1/2}(F)$.

%

Define the function spaces for the fluid velocity $\bu$, the fluid pressure $p$, and the displacement $\beeta$  as
\begin{align*}
& \textbf{X} :=  \{\bv \in \bH^{1}(\Omega_f):\; \bv = {\textbf{0}} \ \ \mbox{  on } \Gamma_f  \}, \\
&Q := L^2(\Omega_f), \\
& \bSigma :=   \{\bxi \in \bH^{1}(\Omega_s): \, \bxi =\b0  \ \ \mbox{  on } \Gamma_s  \}.
\end{align*}
The spaces $\bX$ and $Q$ satisfy the inf-sup condition,
\begin{equation}\label{infsup}
\inf_{q \in Q} \sup_{\bv \in \textbf{X}}\frac{(q, \nabla \cdot \bv)}{\|q\|_{\Omega_f}\|\nabla \bv\|_{\Omega_f}} \geq \beta >0.
\end{equation}
We also define the {\it div-free} space for the fluid velocity, 
$$
\textbf{V}:=\{ \bv \in \textbf{X}: (q, \nabla \cdot \bv)=0, \forall q \in Q\} \,.
$$
Note that, for $\bv \in \bV$, $\|\nabla \cdot \bv\|_{\Omega_f} = 0$ and, therefore,  
$\nabla \cdot \bv =0$ if $\nabla \cdot \bv \in C^0(\Omega)$. 
The dual spaces $\textbf{X}^{*}$ and $\textbf{V}^{*}$ are endowed with the following dual norms
 $$\|\bw\|_{\textbf{X}^{*}}:=\sup_{\bv \in \textbf{X}}\frac{(\bw,\bv)}{\|\nabla \bv\|_{\Omega_f}}, \quad \|\bw\|_{\textbf{V}^{*}}
 :=\sup_{v \in \textbf{V}}\frac{(\bw,\bv)}{\|\nabla \bv\|_{\Omega_f}} \,.$$
 
  In the following Lemma, we show the equivalence of these norms for functions in $\bV$
with continuous divergence.
  \begin{lemma}\label{VXdualrelation}
Let $\bw \in \bV$ and suppose   $\nabla \cdot \bw \in  C^0(\Omega_f)$. 
Then, there exists $C_*>0$ such that 
$$\|\bw\|_{\bV^{*}} \le \|\bw\|_{\bX^{*}} \le C_*\|\bw\|_{\bV^{*}}.$$
 \end{lemma}
 \begin{proof}
The inequality $\|\bw\|_{\bV^*} \le \|\bw\|_{\bX^*}$ is easily obtained, since $\bV \subset \bX$. For the reversed inequality, consider the orthogonal decomposition of 
$\bv \in \bX$, 
$$\bv = \bv_1 + \nabla q,$$
where $\nabla \cdot \bv_1 =0$ and  $q$ is in a subspace of $ H^2(\Omega_f)$, where the semi-norm $\|\Delta q\|_{\Omega_f}$ is equivalent to $\|q\|_{2,\Omega_f}$ 
 \cite{GiraultRaviart} (Sections 3.1 and 3.3). The decomposition is $L^2$-orthogonal, i.e., 
$ (\by, \nabla q) = 0$ for all $\by \in \bX$ satisfying $\nabla \cdot \by=0$.
%
Also, 
$$
\| \nabla (\nabla q)\|_{\Omega_f} \le  C \|\Delta q \|_{\Omega_f} = C \|\nabla \cdot \bv\|_{\Omega_f} \le C \| \nabla  \bv\|_{\Omega_f}. 
$$
Then, by the triangular inequality,
$$
 \|\nabla \bv_1\|_{\Omega_f} \le  \| \nabla  \bv\|_{\Omega_f} + \| \nabla (\nabla q)\|_{\Omega_f} \le C_* \| \nabla \bv\|_{\Omega_f}.
 $$
This estimate and 
the orthogonality give
$$
\|\bw\|_{\bX^*} = \sup_{\bv \in \bX} \frac{(\bw,\bv)}{\|\nabla \bv\|_{\Omega_f}} = \sup_{\bv \in \bX} \frac{(\bw,\bv_1 +\nabla q)}{\|\nabla \bv\|_{\Omega_f}}
\le \sup_{\bv \in \bX} C_*\frac{(\bw,\bv_1)}{\|\nabla \bv_1\|_{\Omega_f}} =  C_* \sup_{\bz \in \bV} \frac{(\bw,\bz)}{\|\nabla \bz\|_{\Omega_f}} =  C_* \|\bw\|_{\bV^*} $$
Thus, we have $\|\bw\|_{\bX^{*}}  \le C_* \|\bw\|_{\bV^{*}}.$
 \end{proof}
 The variational formulation for the fluid-structure system \eqref{moment}-\eqref{poroin} is given by: given the initial conditions, 
 find $(\bu,p,\beeta) \in (\bX, Q,\bSigma)$, for a.e. $t \in (0,T)$, such that  \vspace{-0.2cm}
\begin{eqnarray}
& &  \rho_f \left( \partial_t \bu, \, \bv \right)
 + (2\nu_{f} D(\bu) ,D(\bv))  -( p,\nabla \cdot
\bv )  
\nn \\ & & \hspace{2in}
 =  ( \bbf_f, \bv ) + \left\langle \bsigma_f\bn_f, \bv \right\rangle_{\Gamma} \quad \forall \bv \in \textbf{X} \,, \label{vs1o}
  \\
& &  (q, \nabla \cdot \bu) = 0 \qquad \forall q \in Q \, , \label{vs2o} \\
&& \rho_s\left( \partial_t^2 \beeta, \, \bxi \right) + 2 \nu_s (D(\beeta),D(\bxi)) + \lambda (\nabla \cdot \beeta, \nabla \cdot \bxi)
  \nn \\ & & \hspace{2in}
 =  (\bbf_s, \bxi) +  \langle \bsigma_s\bn_s, \bxi\rangle_{\Gamma}
 \quad \forall \bxi \in  \bSigma. \label{vstr1o}
  \end{eqnarray}

If a normal stress function satisfying $\bsigma_f \bn_f = -\bsigma_s \bn_s$ on $\Gamma$ is given as Neumann conditions for the fluid and structure problems, each local problem is well-posed. In our approach, the normal stress function is represented by an unknown Lagrange multiplier and will be used to impose the continuity of the velocities on the interface.

\section{Global-in-time DD schemes} \label{sec:SWR}
This section discusses time-dependent interface problems for the fluid-structure system, from which global-in-time domain decomposition methods are developed. 
\subsection{Time-dependent Steklov-Poincar\'{e} operator}
We introduce the Lagrange multiplier $\bg \in \bLambda := \bH^{-1/2}(\Gamma)$ representing
\begin{equation}
    \bg :=  \bsigma_f\bn_f = -\bsigma_s\bn_s \quad \mbox{on } \Gamma \times (0,T) \,.
\end{equation}
The equations \eqref{vs1o}-\eqref{vstr1o} are then rewritten as 
\begin{eqnarray}
 & & \rho_f \left( \partial_t \bu, \, \bv \right)
 + (2\nu_{f} D(\bu) ,D(\bv))  -( p,\nabla \cdot
\bv )  \nn \\ & & \hspace{2in}
 =  ( \bbf_f, \bv ) + \left\langle \bg, \bv \right\rangle_{\Gamma} \quad \forall \bv \in \textbf{X} \,, \label{vs1}
  \\
  & & (q, \nabla \cdot \bu) = 0 \qquad \forall q \in Q \, , \label{vs2} \\
 & & \rho_s\left( \partial_t^2 \beeta, \, \bxi \right) + 2 \nu_s (D(\beeta),D(\bxi)) + \lambda (\nabla \cdot \beeta, \nabla \cdot \bxi)
  \nn \\ & & \hspace{2in}
 =  (\bbf_s, \bxi) -  \langle \bg, \bxi\rangle_{\Gamma}
 \quad \forall \bxi \in  \bSigma. \label{vstr1}
  \end{eqnarray}
Let $\bLambda^{*}$ denote the dual space of $\bLambda$ and  
define the following interface operators: for given $\bf_f, \bf_s$, $\bu^0$, $\beeta^{0}$, $\bar{\beeta}^{0}$
$$S_{f}: \bL^2(0,T; \bLambda) \longrightarrow \bL^2(0,T; \bLambda^{*}), \quad  S_{f}(\bg) = \bu(\bg)|_{\Gamma} \, , \,$$
$$S_{s}: \bL^2(0,T; \bLambda) \longrightarrow \bL^2(0,T; \bLambda^{*}), \quad 
S_{s}(\bg) = - \frac{\partial \beeta(\bg)}{\partial t}|_{\Gamma},$$
where $\bu(\bg)$ and $\beeta(\bg)$ are the solutions to the Stokes problem (\ref{vs1})-(\ref{vs2}) and the structure problem (\ref{vstr1}).
 In (\ref{vs1})-(\ref{vstr1}), the interface condition \eqref{inter1} has been imposed  through the use of a common $\bg$, however, the interface condition (\ref{inter2}) is not enforced.
 Therefore, the remaining condition  (\ref{inter2}) leads to the following time-dependent interface problem:  \\
\textit{for a.e $t \in (0,T),$ find $\bg(t) \in \bL^2(0,T,\bLambda)$ satisfying 
\begin{equation}\label{sp1}
\int_0^T \left( \langle S_{f}(\bg), \bv \rangle + \langle S_{s}(\bg), \bv \rangle \right)ds = 0 \quad  \forall \, \bv \in \bL^2(0,T,\bLambda).
\end{equation}
}
The evolutionary interface problem \eqref{sp1} can be solved using iterative methods, e.g., a Krylov method.

\begin{remark} \label{remark:nonsymmetric}
Since the time-dependent Steklov–Poincaré operators are nonsymmetric, the standard methods cannot be implemented to prove the existence and uniqueness of solutions to the space-time interface problem \eqref{sp1}. Thus, addressing this issue is an open question. Also see remark 1 of \cite{THe}.
\end{remark}

 \begin{remark} \label{remark:nonlinear}
If the nonlinear Stokes or the Navier-Stokes equations are considered for the FSI system,
the interface operator is nonlinear, and an iteration formula for the nonlinear problem is defined using the linearized fluid equations. See \cite{THe}. 
\end{remark}

\subsection{Robin transmission conditions and the space-time interface problem}\label{subsec:robin}
The two-sided Robin interface conditions on $\Gamma$ are established by linearly combining equations (\ref{inter1}) and (\ref{inter2}) with coefficients of $(\alpha_f, 1)$ and $(-\alpha_s, 1)$, respectively, where $\alpha_f,\alpha_s > 0$ \cite{gig}:
\begin{equation}\label{R1}
    \bg_f:=\alpha_f \bu + \bsigma_f\bn_f = \alpha_f\partial_t \beeta -\bsigma_s\bn_s \quad \mbox{on } \Gamma \times (0,T)  \,,
\end{equation}
\begin{equation}\label{R2}
    \bg_s:=-\alpha_s\partial_t \beeta -\bsigma_s\bn_s = -\alpha_s \bu + \bsigma_f\bn_f  \quad \mbox{on } \Gamma \times (0,T) \,.
\end{equation}
\begin{remark}[Regularity of Normal stress] 
 With the function spaces defined in Section 2, the stress functions $\bsigma_f\bn_f$, $\bsigma_s\bn_s$ in (3.6)-(3.7) are in 
$\bH^{-1/2}(\Gamma)$. However, in order to formulate the DD scheme as an interface problem and analyze it, we need the interface functions $\bg_f$, and $\bg_s$ to have
 $L^2$-regularity.  
Therefore, we assume that the weak formulations \eqref{fl1}-\eqref{st1} and \eqref{fsi_dis_Stokes_1}-\eqref{fsi_dis_str_3}, considered for the SWR algorithm
 in the subsequent sections, are well-posed with sufficient solution regularity so that $\bsigma_f\bn_f$ and $\bsigma_s\bn_s$ have
the $L^2$ regularity in $\Gamma$. 
Unfortunately, we have not yet established or come across a proof of these regularity results in the literature. Nonetheless, they have been assumed in several studies (see Remark 3.1 of \cite{A}).\label{regularity}
\end{remark}

The regularity assumption in Remark \ref{regularity} allows the unknown functions $\bg_f$ and $\bg_s$ to have $L^2$-regularity in space, which is needed to define interface operators for the variables and analyze the DD algorithm discussed later in this section.

If we let $\bg_f$ be a Robin condition for the Stokes equations with the parameter $\alpha_f >0$ as in the left-hand side of \eqref{R1}, 
the corresponding weak formulation is given as follows: find $(\bu,p) \in (\bX, Q)$, for a.e. $t \in (0,T)$, such that
\begin{eqnarray}
  & &   \rho_f( \partial_t \bu,\bv) +2\nu_f(D(\bu),D(\bv)) -(p,\nabla \cdot \bv) \nonumber \\
    & & \hspace{1in}+ \alpha_f(\bu, \bv)_{\Gamma} \,  = \, (\bf_f,\bv)+ (\bg_f, \bv)_{\Gamma} \quad \forall \bv \in \textbf{X} \,,  \label{vs1pcovph}  \\
  & &    (q,( \nabla \cdot \bu) \, = \, 0 \quad \forall q \in Q \,. \label{vs2pcovph} 
 \end{eqnarray}
 Similarly, considering $\bg_s$  as a Robin condition for the elastic system with the parameter $\alpha_s>0$ as in \eqref{R2}, 
 we have the weak formulation given by:
 find $\beeta \in \bSigma$, for a.e. $t \in (0,T)$ satisfying
 
 \begin{eqnarray}
 \rho_s\left(\partial_t^2 \beeta, \, \bxi \right) + 2 \nu_s (D(\beeta),D(\bxi)) + \lambda (\nabla \cdot \beeta, \nabla \cdot \bxi) + \alpha_s(\partial_t\beeta, \bxi)_{\Gamma} \nn \\
 =  (\bbf_s, \bxi) -  ( \bg_s ,\bxi)_{\Gamma},
 \quad \forall \bxi \in  \bSigma.\label{wkst}
 \end{eqnarray}

Denote by $(\bu,p)=\left (\bu(\bg_{f},\bf_{f},\bu_{0}), p(\bg_{f},\bf_{f},\bu_{0})\right )$ the solution to the Stokes problem~\eqref{vs1pcovph}-\eqref{vs2pcovph}, and $\beeta=\beeta(\bg_{s},\bf_{s},\beeta_{0}, \bar{\beeta}^0)$ the solution to the structure problem~\eqref{wkst}. To derive the interface problem associated with the Robin conditions \eqref{R1}-\eqref{R2}, we first define the interface operator:
$$
{\cal R} : \left(L^2(0,T; \, \bL^2(\Gamma)) \right)^2  \rightarrow   \left(L^2(0,T; \, \bL^2(\Gamma)) \right)^2,
$$
such that 
\begin{equation} \label{R_def}
{\cal R} \left[ \begin{array}{c} \bg_f  \\ \bg_s  \end{array} \right]  = \left[ \begin{array}{c} \bg_s + (\alpha_s+\alpha_f)  \left( \partial_t \beeta (\bg_{s},\bf_{s},\beeta_{0},\bar{\beeta}_0)\right) \mid_{\Gamma} 
\vspace{4pt} \\ 
\bg_f - (\alpha_f+\alpha_s)  \left(  \bu (\bg_{f},\bf_{f},\bu_{0}) \right) \mid_{\Gamma}  \end{array} \right] \,.
\end{equation}
The Robin transmission conditions \eqref{R1}-\eqref{R2} are then equivalent to the following space-time interface problem for two interface variables $\bg_f$ and $\bg_s$:
\begin{equation} \label{ifprob}
{\cal S_R} \left[ \begin{array}{c} \bg_f  \\ \bg_s  \end{array} \right] = \chi_{\cal R} \quad \text{on} \; \Gamma \times (0,T),
\end{equation}
where 
$$
{\cal S_{R}} \left[ \begin{array}{c} \bg_f  \\ \bg_s  \end{array} \right] = \left[ \begin{array}{c} \bg_f  \\ \bg_s  \end{array} \right] - \left[ \begin{array}{c} \bg_s + (\alpha_s+\alpha_f)  \left( \partial_t \beeta (\bg_{s},\textbf{0},\textbf{0},\textbf{0})\right) \mid_{\Gamma} 
\vspace{4pt} \\ 
\bg_f - (\alpha_f+\alpha_s)  \left(  \bu (\bg_{f},\textbf{0},\textbf{0}) \right) \mid_{\Gamma}  \end{array} \right]
$$
and
$$ 
\chi_{\cal R}=\left[ \begin{array}{c} (\alpha_s+\alpha_f)  \left( \partial_t \beeta (\textbf{0},\bf_{s},\beeta_{0},\bar{\beeta}_0)\right) \mid_{\Gamma} 
\vspace{4pt} \\ 
- (\alpha_f+\alpha_s)  \left(  \bu (\textbf{0},\bf_{f},\bu_{0}) \right) \mid_{\Gamma}  \end{array} \right].
$$  
The weak form of \eqref{ifprob} reads as: find $(\bg_{f},\bg_{s}) \in  \left(L^2(0,T; \, \bL^2(\Gamma)) \right)^2  $, for a.e. $t \in (0,T)$, such that 
\begin{equation} \label{ifprob_weak}
\int_0^T \int_{\Gamma} \left ({\cal S_{R}} \left[ \begin{array}{c} \bg_f  \\ \bg_s  \end{array} \right] \cdot  \left[ \begin{array}{c} \bxi_{f}  \\ \bxi_s  \end{array} \right]  \right ) \, d\gamma \, dt  = \int_0^T \int_{\Gamma} \left (\chi_{\cal R} \cdot  \left[ \begin{array}{c} \bxi_{f}  \\ \bxi_p  \end{array} \right]\right ) \, d\gamma \, dt   
\end{equation}
$\forall \left (\bxi_{f}, \bxi_{p}\right ) \in  \left(L^2(0,T; \, \bL^2(\Gamma)) \right)^2$. 

The interface problem \eqref{ifprob_weak} can be solved by iterative methods such as GMRES and simple Jacobi-type methods. We consider a Schwarz waveform relaxation (SWR) algorithm 
based on Robin transmission conditions and show the convergence of the algorithm.

\subsubsection{Schwarz waveform relaxation (SWR) algorithm}\label{subsec:SWR}
Consider the following SWR algorithm based on Robin transmission conditions: at the $k$th iteration step we solve
\begin{eqnarray}
\rho_f\partial_t \bu^{k} -\nabla \cdot (2\nu_f D(\bu^{k})-p^{k} \bI)&=&\bf_f \quad  \mbox{in } \Omega_f \times (0,T)\,, \label{s1o} \\
 \nabla \cdot \bu^{k} &=&0 \quad \mbox{in } \Omega_f \times (0,T)\,, \label{s2o}
  \\
 \alpha_f \bu^{k} + \bsigma^{k}_f\bn_f &=& \alpha_f\partial_t \beeta^{k-1} -\bsigma^{k-1}_s\bn_s  \quad \mbox{on } \Gamma \times (0,T)\,, \label{s4o}
 \end{eqnarray}
for  $(\bu^{k}, p^{k})$ satisfying the initial and boundary conditions \eqref{fluidin}, \eqref{fluidbd}, and
\begin{eqnarray}
\rho_s \partial_t^2 \beeta^{k} -2 \nu_s \, \nabla \cdot D(\beeta^{k})  - \lambda \nabla (\nabla \cdot \beeta^{k}) &=& \bbf^{k}_s \quad \mbox{in } \Omega_s \times (0, T) \,,\\
 -\alpha_s\partial_t \beeta^{k} -\bsigma^{k}_s\bn_s &=& -\alpha_s \bu^{k-1} + \bsigma^{k-1}_f\bn_f  \quad \mbox{on } \Gamma \times (0,T) \label{d4o}
 \end{eqnarray}
for $\beeta^{k}$ satisfying \eqref{porobd}-\eqref{poroin}.
The weak formulation of this decoupled system is written as:  at the $k$th iteration, find $(\bu^{k},p^{k},\beeta^{k}) \in (\bX, Q, \bSigma)$, for a.e. $t \in (0,T)$, such that
\begin{eqnarray}
 & & \rho_f \left( \partial_t \bu^{k}, \, \bv \right)
 + (2\nu_{f} D(\bu^{k}) ,D(\bv))  -( p^{k},\nabla \cdot
\bv ) +  \alpha_f(\bu^{k},\bv)_{\Gamma}  \nn \\ & & \hspace{.8in}
 =  ( \bbf_f, \bv ) +  \left( \alpha_f\partial_t \beeta^{k-1} -\bsigma^{k-1}_s\bn_s, \bv \right)_{\Gamma} \qquad \forall \bv \in \textbf{X} \,, \label{fl1}
  \\[1.5ex]
  & & (q, \nabla \cdot \bu^{k}) = 0 \qquad \forall q \in Q\,, \label{fl2} \\[1.5ex]
 & & \rho_s\left( \partial_t^2 \beeta^{k}, \, \bxi \right) + 2 \nu_s (D(\beeta^{k}),D(\bxi)) + \lambda (\nabla \cdot \beeta^{k}, \nabla \cdot \bxi) + \alpha_s(\partial_t \beeta^{k}, \bxi)_{\Gamma} \nn \\ & & \hspace{.5in}
 =  (\bbf_s, \bxi) -  ( -\alpha_s \bu^{k-1} + \bsigma^{k-1}_f\bn_f ,\bxi)_{\Gamma},
 \quad \forall \bxi \in  \bSigma. \label{st1}
 \end{eqnarray}


In the next theorem, we prove the convergence of the proposed algorithm. The idea of proving the convergence of the Robin-Robin method by energy estimates comes from \cite{Lion}. Then, it was adapted to the time-dependent problems in \cite{Martin}, and also used in \cite{HJJ13,THH}. We apply the same approach to our problem. 

\begin{theorem}
Suppose $\bf_f \in \textbf{X}^{*}$, $\bf_s \in \bSigma^{*}$ and $\alpha_s \geq \alpha_f > 0.$ If an initial  $(\bu^0, \beeta^0,\bar{\beeta}^0)$ is chosen such that the Robin-Robin conditions (\ref{R1}), (\ref{R2}) are well-defined in $\bL^2({\Gamma})$ for a.e. $t \in (0,T)$, then 
 the weak formulation
 (\ref{fl1})-(\ref{st1}) generates a convergent sequence of iterates
 $$(\bu^{k},\beeta^{k})\in L^{\infty}(0,T;\textbf{X})\times L^{\infty}(0,T;\bSigma).$$ 
With the additional assumption that $\bsigma_s \in \bH^1(\Omega_s)$ and $\nabla \cdot \bu \in C^0(\Omega_f)$, the pressure $p^{k}$ also converges in $L^{2}(0,T;Q)$.
  \label{theo1}
\end{theorem}
\begin{proof}
Since the equations \eqref{fl1}-\eqref{st1} are linear, we show that the iterate $(\bu^{k},p^{k},\beeta^{k})$  converges to zero in suitable norms by setting $\bf_f = \bu^0 = \textbf{0}$ and $\bf_s = \beeta^0=\bar{\beeta}^0=\textbf{0}$.
The proof is organized as follows: we begin by applying identities \eqref{sq1} and \eqref{sq21} in the weak formulation of both the fluid and structure subproblems respectively. These identities are chosen to create matching terms on both sides of \eqref{hkineq}, which will cancel each other out and result in a constant when summed across iterations. Next, by combining the results with the use of the Robin conditions and applying the Gronwall lemma, we establish the convergence of the velocity and displacement in respective norms. Finally, we utilize the energy estimates of velocity and displacement, along with \Cref{VXdualrelation}, to establish the convergence of the pressure.
Taking $\bv=\bu^k$ and $q=p^k$ in \eqref{fl1} and \eqref{fl2}, we get
\begin{eqnarray}\label{ns1}
&&\rho_f \left( \partial_t \bu^{k}, \, \bu^{k} \right)
 + (2\nu_{f} D(\bu^{k}) ,D(\bu^{k}))+ \alpha_f(\bu^{k},\bu^{k})_{\Gamma}  \nn \\ & & \hspace{.8in}
 = \left( \alpha_f\partial_t \beeta^{k-1} -\bsigma^{k-1}_s\bn_s, \bu^{k} \right)_{\Gamma}.
\end{eqnarray}
Using the identity
\begin{equation}\label{sq1}
    (\bsigma_f^k\bn_f + \alpha_f \bu^k)^2 - (\bsigma_f^k\bn_f - \alpha_s \bu^k)^2 = 2(\alpha_f + \alpha_s)(\bsigma_f^k\bn_f)(\bu^k)+(\alpha_f^2-\alpha_s^2)(\bu^k)^2
\end{equation}
and the Robin condition \eqref{s4o},  we can rewrite \eqref{ns1} as 
\begin{eqnarray}
 & & \rho_f \left( \partial_t \bu^k, \, \bu^k \right)
 +2\nu_{f} \norm{D(\bu^k)}_{\Omega_f}^2 +\frac{1}{2(\alpha_f + \alpha_s)}\int_{\Gamma}(\bsigma_f^k\bn_f - \alpha_s \bu^k)^2 \; d\gamma  \nn \\ & &  \hspace{.2in}
 =  \frac{1}{2(\alpha_f + \alpha_s)}\int_{\Gamma}(\bsigma_f^k\bn_f + \alpha_f\bu^k)^2 \; d\gamma+\frac{1}{2}(\alpha_s - \alpha_f)\int_{\Gamma}(\bu^k)^2 \; d\gamma
  \nn \\ & & \hspace{.2in}
 =  \frac{1}{2(\alpha_f + \alpha_s)}\int_{\Gamma}(-\bsigma_s^{k-1}\bn_s + \alpha_f \partial_t \beeta^{k-1})^2 \;d\gamma+\frac{1}{2}(\alpha_s - \alpha_f)\int_{\Gamma}(\bu^k)^2 \; d\gamma. 
\end{eqnarray}
We then integrate over $(0,t)$ for a.e. $t\in (0,T]$. And, since $\bu^k \in \bH^{1}(\Omega_f)$, apply the trace theorem and Young's inequality to obtain, 
\begin{eqnarray}
 & & \frac{\rho_f}{2} \norm{\bu^k}^2_{\Omega_f}
 +2\nu_{f} \int_0^t \norm{D(\bu^k)}_{\Omega_f}^2ds +\frac{1}{2(\alpha_f + \alpha_s)}\int_0^t \int_{\Gamma}(\bsigma_f^k\bn_f - \alpha_s \bu^k)^2  d\gamma\, ds  \nn \\ & & 
 \leq  \frac{1}{2(\alpha_f + \alpha_s)}\int_0^t \int_{\Gamma}(-\bsigma_s^{k-1}\bn_s + \alpha_f \partial_t \beeta^{k-1})^2 d\gamma\, ds + \overline{C}\int_0^{t} \norm{\bu^k(s)}_{\Omega_f}\norm{D(\bu^k(s))}_{\Omega_f} d\gamma \,ds  \nn \\ & &
  \leq  \frac{1}{2(\alpha_f + \alpha_s)}\int_0^t \int_{\Gamma}(-\bsigma_s^{k-1}\bn_s + \alpha_f \partial_t \beeta^{k-1})^2 d\gamma \, ds \nn \\ & &
 \hspace{.5in} + \overline{C}\int_0^{t} \left(\frac{1}{4\epsilon}\norm{\bu^k(s)}^2_{\Omega_f}+\epsilon \norm{D(\bu^k(s))}^2_{\Omega_f}\right)ds
\end{eqnarray}
for some constant $\overline{C}>0$ and $\epsilon>0$. Choosing $\epsilon = \nu_f/\overline{C}$, we have
\begin{eqnarray}\label{k1}
 & & \frac{\rho_f}{2} \norm{\bu^k}^2_{\Omega_f}
 +\nu_{f} \int_0^t \norm{D(\bu^k)}_{\Omega_f}^2ds +\frac{1}{2(\alpha_f + \alpha_s)}\int_0^t \int_{\Gamma}(\bsigma_f^k\bn_f - \alpha_s \bu^k)^2 d\gamma \, ds  \nn \\ & & 
 \leq \frac{1}{2(\alpha_f + \alpha_s)}\int_0^t \int_{\Gamma}(-\bsigma_s^{k-1}\bn_s + \alpha_f \partial_t \beeta^{k-1})^2 d\gamma\, ds +C\int_0^{t} \norm{\bu^k(s)}^2_{\Omega_f}ds,
\end{eqnarray}
where $C = \overline{C}^2/(4\nu_f)$.

Now, similarly for the structure part, taking $\bxi = \partial_t \beeta^k$ in \eqref{st1},
\begin{eqnarray}\label{s1}
 & & \rho_s\left( \partial_t^2 \beeta^k, \, \partial_t \beeta^k \right) + 2 \nu_s (D(\beeta^k),D(\partial_t \beeta^k)) + \lambda (\nabla \cdot \beeta^k, \nabla \cdot \partial_t \beeta^k) + \alpha_s(\partial_t \beeta^{k},  \partial \beeta^{k})_{\Gamma} \nn \\ & & \hspace{.5in}
 =  -( -\alpha_s \bu^{k-1} + \bsigma^{k-1}_f\bn_f,\partial_t \beeta^k)_{\Gamma}.
\end{eqnarray}
Using the identity
\begin{equation}\label{sq21}
    (-\bsigma_s^k\bn_s - \alpha_s \partial_t \beeta^k)^2 - (-\bsigma_s^k\bn_s + \alpha_f \partial_t \beeta^k)^2 = 2(\alpha_f + \alpha_s)(\bsigma_s^k\bn_s)(\partial_t \beeta^k)+(\alpha_s^2-\alpha_f^2)(\partial_t \beeta^k)^2
\end{equation}
and the Robin condition \eqref{d4o}, \eqref{s1} implies 
\begin{eqnarray}
 & & \rho_s\left( \partial_t^2 \beeta^k, \, \partial_t \beeta^k \right) + 2 \nu_s (D(\beeta^k),\partial_t D(\beeta^k)) + \lambda (\nabla \cdot \beeta^k, \partial_t \nabla \cdot \beeta^k) \nn \\ & & + \frac{\alpha_s - \alpha_f}{2} \int_{\Gamma}(\partial_t \beeta^k)^2 \; d\gamma + \frac{1}{2(\alpha_f + \alpha_s)} \int_{\Gamma}(-\bsigma_s^k\bn_s + \alpha_f \partial_t \beeta^k)^2 \; d\gamma \nn \\ & & \hspace{.5in}
 \leq   \frac{1}{2(\alpha_f + \alpha_s)} \int_{\Gamma}(-\bsigma_s^k\bn_s - \alpha_s \partial_t \beeta^k)^2 \; d\gamma.
\end{eqnarray}
Integrate over $(0,t)$ for a.e. $t\in (0,T)$ and apply the Robin boundary conditions \eqref{d4o} to obtain
\begin{eqnarray}\label{k2}
& & \frac{\rho_s}{2}\norm{\partial_t \beeta^k}^2_{\Omega_s} + \nu_s \norm{D(\beeta^k)}^2_{\Omega_s} + \frac{\lambda}{2} \norm{\nabla \cdot \beeta^k}^2_{\Omega_s}+ \frac{\alpha_s - \alpha_f}{2} \int_0^t \norm{\partial_t \beeta^k}^2_{\Gamma} \; ds  \nn \\ & & \hspace{.5in}
+ \frac{1}{2(\alpha_f + \alpha_s)} \int_0^t \int_{\Gamma}(-\bsigma_s^k\bn_s + \alpha_f \partial_t \beeta^k)^2 \;d\gamma\, ds \nn \\ & & \hspace{.5in}
 \leq  \frac{1}{2(\alpha_f + \alpha_s)} \int_0^t \int_{\Gamma}(\bsigma_f^{k-1}\bn_f - \alpha_s \bu^{k-1})^2 \; d\gamma\,ds.
\end{eqnarray}
 Define, for all $k \geq 1$ and for a.e. $t\in (0,T]$,
\begin{eqnarray*}
    E^{k}(t)& := &\frac{\rho_f}{2}\norm{\bu^k(t)}_{\Omega_f}^2+ \nu_f\int_{0}^t \norm{D(\bu^k)}_{\Omega_f}^2 \;ds+\frac{\rho_s}{2} \norm{\partial_t\beeta^k }_{\Omega_s}^2 \; \\
    && + \nu_s \norm{D(\beeta^k)}^2_{\Omega_s} + \frac{\lambda}{2} \norm{\nabla \cdot \beeta^k}^2_{\Omega_s} + \frac{\alpha_s - \alpha_f}{2} \int_0^t \norm{\partial_t \beeta^k}^2_{\Gamma}ds, \\
    B^k(t) &:= &  \frac{1}{2(\alpha_f+\alpha_s)} \int_0^t \norm{\bsigma_f^k\bn_f - \alpha_s \bu^k}_{\Gamma}^2  \; ds
    \\
    &&+ \frac{1}{2(\alpha_f+\alpha_s)} \int_0^t \norm{-\bsigma_s^k\bn_s + \alpha_f \partial_t \beeta^k}_{\Gamma}^2 \; ds \, ,
\end{eqnarray*}
where $\frac{\alpha_s - \alpha_f}{2}\ge0$ with the assumption $\alpha_s \ge \alpha_f$.
Adding \eqref{k1} to \eqref{k2}
yields 
\begin{equation}\label{hkineq}
   E^k(t)+B^k(t)\leq B^{k-1}(t)+C \int_0^t\|\bu^k(s)\|_{\Omega_f}^2 \;ds \, , 
\end{equation}
and summing over the iterates for any given $K>0$, we obtain
 \begin{equation} \label{resultforgron}
     \sum_{k=1}^K E^k(t)\leq B^0 (t)+C \sum_{k=1}^K \int_0^t\|\bu^k(s)\|_{\Omega_f}^2 \; ds. 
 \end{equation}
 In \eqref{resultforgron}  $$B^0(t)=\frac{1}{2(\alpha_f+\alpha_s)} \int_0^t g_0 \, ds, $$
where $g_0 =\norm{\bsigma_f^0\bn_f - \alpha_s \bu^0}_{\Gamma}^2 +\norm{-\bsigma_s^0\bn_s + \alpha_f \partial_t \beeta^0}_{\Gamma}^2$  is obtained by the initial guess. 
Now, the definition of $E^k(t)$ and  (\ref{resultforgron}) yield 
$$\frac{\rho_f}{2}\sum_{k=1}^K \|\bu^k(t)\|_{\Omega_f}^2 \leq B^0 (t)+C \sum_{k=1}^K \int_0^t \|\bu^k(s)\|_{\Omega_f}^2 \;ds,$$
and applying Gronwall's lemma, we obtain 
\begin{equation}\label{convergenceinu}
    \sum_{k=1}^K \|\bu^k(t)\|_{\Omega_f}^2 \leq \frac{2B^0(T)}{\rho_f}e^{\frac{2C T}{\rho_f}}
\end{equation}
for any $K>0$ and a.e. $t\in(0,T)$. 
The inequality (\ref{convergenceinu}) implies that $\bu^k$ converges to 0 in $L^{\infty}(0,T;\mathbf{L^2}(\Omega_f))$ as $k \rightarrow \infty$. Also,  
the inequalities (\ref{resultforgron}) and (\ref{convergenceinu}) yield
\begin{eqnarray}
 &  &\sum_{k=1}^K \left(\nu_f\int_{0}^t \norm{D(\bu^k)}_{\Omega_f}^2 ds+\frac{\rho_s}{2} \norm{\partial_t\beeta^k }_{\Omega_s}^2 + \nu_s \norm{D(\beeta^k)}^2_{\Omega_s} + \frac{\lambda}{2} \norm{\nabla \cdot \beeta^k}^2_{\Omega_s}  \right.\nn \\
 & &  \left. \hspace{.2in} 
+ \frac{\alpha_s-\alpha_f}{2} \int_0^t \norm{\partial_t \beeta^k}^2_{\Gamma}ds \right)
    \, \leq  \, \left(1+\frac{2CT}{\rho_f}e^{\frac{2C T}{\rho_f}}\right)B^0(T), \label{convine2}
\end{eqnarray}
 which implies that  $D(\bu^k)$,  $\partial_t \beeta^k$,  $D(\beeta^k)$, $\nabla \cdot \beeta^k$ converge to $0$ in $L^{2}(0,T;\textbf{L}^2(\Omega_f))$, $L^{\infty}(0,T;\textbf{L}^2(\Omega_s))$, $L^{\infty}(0,T;\textbf{L}^2(\Omega_s))$,
 $L^{\infty}(0,T;L^2(\Omega_s))$,  respectively, as $k \rightarrow \infty$. In addition, $\bsigma_s^k\bn_s$ converges to 0 in $L^{\infty}(0,T;\textbf{L}^2(\Omega_s))$  by its definition, and using Poincar\'{e}-Friedrichs inequality, the convergence of $\beeta^k$ to $0$ in $L^{\infty}(0,T;\textbf{L}^2(\Omega_s))$
is obtained.

By the Robin condition (\ref{R2}) with $\bg_s \in \bL^2(\Gamma)$, the trace $\bsigma_s\bn_s$ on $\Gamma$ is in $\bL^2(\Gamma)$ for a.e. $t\in(0,T)$. We prove the convergence of 
$p^{k}$ using same approach mentioned in \cite{phlee} with additional regularity assumption that $\bsigma_s$ is $ \bH^1(\Omega_s)$ for a.e. $t\in(0,T)$.
First, we estimate a bound for the time derivative term in \eqref{fl1}. For $\bv \in \textbf{V}$ the equation \eqref{fl1} is written as 
 \begin{eqnarray}
     \rho_f(\partial_t \bu^k,\bv)
     &= & -2\nu_f(D(\bu^{k}),D(\bv)) -\alpha_f ( \bu^{k},\bv)_{\Gamma}+ (\alpha_f \partial_t \beeta^{k-1} - \bsigma_s^{k-1} \bn_s, \bv)_{\Gamma} \,     \label{hs1}
\end{eqnarray}
Then, using Cauchy-Schwarz inequality, we have for some constants $C_1, C_2, C_3 >0,$
 \begin{eqnarray}
     \rho_f(\partial_t \bu^k,\bv)
     &\leq & C_1\|D(\bu^k)\|_{\Omega_f}\|D(\bv)\|_{\Omega_f} + C_2\|\bu^k\|_{\Gamma}\|\bv\|_{\Gamma} \nn \\
     &&+ C_3 (\|\partial_t \beeta^{k-1}\|_{\Gamma}
+ \|\bsigma_s^{k-1}\bn_s\|_{\Gamma})\|\bv\|_{\Gamma} \, ,    \label{hs2}
\end{eqnarray}
and, using the trace theorem,
 \begin{eqnarray}
     \rho_f(\partial_t \bu^k,\bv)
     &\leq & C_1\|D(\bu^k)\|_{\Omega_f}\|D(\bv)\|_{\Omega_f} + C_{2,T}\|\bu^k\|_{1, \Omega_f}\|\bv\|_{1, \Omega_f} \nn \\
     &&+ C_{3,T} (\|\partial_t \beeta^{k-1}\|_{\Gamma}
+ \|\bsigma_s^{k-1}\bn_s\|_{\Gamma})\|\bv\|_{1, \Omega_f}     \label{hs3}
\end{eqnarray}
for some constants $C_{2,T}, C_{3,T} > 0$.
Then Korn's inequality implies, for some constants $C_K>0$,
 \begin{eqnarray}
     \rho_f(\partial_t \bu^k,\bv)
     &\leq & C_{K}\left(\|D(\bu^k)\|_{\Omega_f} + \|\partial_t \beeta^{k-1}\|_{\Gamma}
+ \|\bsigma_s^{k-1}\bn_s\|_{\Gamma} \right)\|\nabla \bv\|_{\Omega_f} \, .    \label{hs}
\end{eqnarray}

Now using Poincar\'{e}-Friedrichs inequality, dividing both sides by $\norm{\nabla \bv}_{\Omega_f}$ and taking supremum over $\bv \in \textbf{V}$, we have, for some constant $\widehat{C} > 0$,  
\begin{equation*}\label{vs1pcovppp}
  \rho_f   \| \partial_t \bu^k\|_{\textbf{V}^{*}} \leq  \widehat{C} ( \,  \|D(\bu^k)\|_{\Omega_f}+ \|\partial_t \beeta^{k-1}\|_{\Gamma}
+ \|\bsigma_s^{k-1}\bn_s\|_{\Gamma}  \, ) \, .
 \end{equation*} 
 The norm equivalence of $\|\cdot\|_{\textbf{X}^{*}}$ and $\|\cdot\|_{\textbf{V}^{*}}$ (see Lemma \ref{VXdualrelation})
then implies, for some constant $C_{*}>0,$
 \begin{equation}\label{vs1pcovpp2}
     \rho_f  \| \partial_t \bu^k\|_{\textbf{X}^{*}} \leq
    C_{*}( \, \|D(\bu^k)\|_{\Omega_f}+ \|\partial_t \beeta^{k-1}\|_{\Gamma}
+ \|\bsigma_s^{k-1}\bn_s\|_{\Gamma}  \, ) \, .
 \end{equation}
 
To estimate a bound for $p^k$,  consider (\ref{fl1}) with $\bv \in \textbf{X}.$ We isolate the pressure term, divide by $\|\nabla \bv\|_{\Omega_f}$, take supremum over $\bv \in \textbf{X}$. Then the inf-sup condition (\ref{infsup}) 
and the estimate (\ref{vs1pcovpp2})
yield 
 \begin{equation*}\label{vs1pcovp3}
     \beta \|p^{k}\|_{\Omega_f} \, \leq \, ( 1+C_{*})( \,  \|D(\bu^k)\|_{\Omega_f}+ \|\partial_t \beeta^{k-1}\|_{\Gamma}
+ \|\bsigma_s^{k-1}\bn_s\|_{\Gamma}  \, ) \, .
 \end{equation*}
 for some $\beta >0$.
 If we square both sides, integrate over the interval $(0,t)$ for a.e. $t\in (0,T)$, then using the trace theorem, 
 \begin{eqnarray}
\frac{\beta^2}{3( 1+C_{*})^2}\int_0^{t} \|p^{k}\|^2_{\Omega_f}  \; ds \, 
& \leq & \int_0^t ( \, \|D(\bu^k)\|^2_{\Omega_f}+ \|\partial_t \beeta^{k-1}\|_{1,\Omega_s}\|\partial_t \beeta^{k-1}\|_{\Omega_s}
\nn \\ & &  + \|\bsigma_s^{k-1}\|^{1/2}_{1,\Omega_s}\|\bsigma_s^{k-1}\|^{1/2}_{\Omega_s}  \, ) \; ds\, .
 \end{eqnarray}
Now, $\|\partial_t \beeta^{k-1}\|_{1,\Omega_s},\|\bsigma_s^{k-1}\|_{1,\Omega_s} < \infty$ as $\partial_t \beeta^{k-1},\bsigma_s^{k-1} \in \bH^1(\Omega_s)$ for a.e. $t\in(0,T)$.  Hence the convergence of 
$D(\bu^k)$,  $\partial_t \beeta^k$, 
 $\bsigma_s^k$ implies that 
 $\displaystyle \int_0^{t}\|p^{k}\|^2_{\Omega_f} ds$ converges to 0 as 
 $k\rightarrow \infty$, i.e., 
 $p^k$ converges to 0 in $L^2(0,T; L^2(\Omega_f))$.
\end{proof}


\section{Nonconforming time discretization and SWR algorithm}
\label{sec:semidiscrete}
 The global-in-time DD approach allows the use of separate time discretizations in each subdomains because the local problems are still time-dependent. On the space-time interface the transfer of information between different time grids is achieved through a suitable projection technique.

Consider $\tau_f$ be a partition of time interval $(0,T)$ into subintervals for the Stokes domain. Let $J_f^m := (t_f^{m-1},t_f^m]$ and step size $\Delta t_f^m := t_f^m-t_f^{m-1}$ for $m=1,.....,M_f$. The space of piecewise constant functions in time on grid $\tau_f$, with values in $W=L^{2}(\Gamma)$, is denoted by $P_0(\tau_f,W)$:
$$P_0(\tau_f,W)=\{\phi:(0,T)\rightarrow W, \phi \, \mbox{is constant on} \, J_f^m \quad \forall m=1,....,M_f\}.$$
Similarly, we define $\tau_s,M_s,J_s^n$ and $ \Delta t_s^n$ for the structure domain. 
To exchange data on the space-time interface between different time grids, we introduce the $L^2$ projection $\Pi_{s,f}$ from $P_0(\tau_f,W)$ onto $P_0(\tau_s,W)$ \cite{THH}: 
$$\Pi_{s,f}(\phi)|_{J_s^n}=\frac{1}{|J_s^n|}\sum_{l=1}^{M_f}\int_{J_s^n\cap J_f^l}\phi. $$
The projection $\Pi_{f,s}$ from $P_0(\tau_s,W)$ onto $P_0(\tau_f,W)$ is also defined similarly. 
\begin{figure}[!h]
\begin{minipage}[b]{0.5\linewidth}
\centering
\tdplotsetmaincoords{70}{30}
\begin{tikzpicture}[tdplot_main_coords]
\coordinate (O) at (0,0,0);
\fill[blue!50,opacity=0.5] (1.5,-3.8,3) -- (1.5,0,3) -- (0,0,3) -- (0,-3.8,3) -- cycle;
\fill[blue!50,opacity=0.5] (1.5,-3.8,3) -- (1.5,0,3) -- (1.5,0,0) -- (1.5,-3.8,0) -- cycle;
\fill[blue!50,opacity=0.5] (1.5,-3.8,3) -- (1.5,-3.8,0) -- (0,-3.8,0) -- (0,-3.8,3) -- cycle;

\fill[magenta!50,opacity=0.5] (3,-3.8,3) -- (3,0,3) -- (1.5,0,3) -- (1.5,-3.8,3) -- cycle;
\fill[magenta!50,opacity=0.5] (3,-3.8,3) -- (3,0,3) -- (3,0,0) -- (3,-3.8,0) -- cycle;
\fill[magenta!50,opacity=0.5] (3,-3.8,3) -- (3,-3.8,0) -- (1.5,-3.8,0) -- (1.5,-3.8,3) -- cycle;

    \draw[->] (O) --++ (3.5,0,0) node[below] {$x$};
    \draw[->] (O) --++ (0,-4.5,0) node[below] {$y$};
    \draw[->] (O) --++ (0,0,3.5) node[right] {$t$};
    \node[] at (0.75,-1.7,0) {$ \textcolor{blue}{\Omega_f} $};
    \node[] at (2.4,-1.7,0) {$ \textcolor{red}{\Omega_s} $};

    \draw[dashed] (1.5,-3.8,0)-- (1.5,0,0);
    \draw[dashed] (1.5,0,0)-- (1.5,0,3);
    
    \draw[-] (0,0,3)-- (0,-3.8,3);
    \draw[-] (1.5,0,3)-- (1.5,-3.8,3);
    \draw[-] (3,0,3)-- (3,-3.8,3);
        
    \draw[-] (1.5,-3.8,3)-- (1.5,-3.8,0);
    \draw[-] (1.5,-3.8,0)-- (3,-3.8,0);
    
    \draw[-] (1.5,-3.8,3)-- (3,-3.8,3);
    \draw[-] (1.5,0,3)-- (3,0,3);
    \draw[-] (0,0,3)-- (1.5,0,3);
    \draw[-] (0,-3.8,3)-- (1.5,-3.8,3);
    \draw[-] (0,-3.8,0)-- (1.5,-3.8,0);
    
    \draw[-] (3,-3.8,0)-- (3,0,0);

    \draw[-] (0,-3.8,0)-- (0,-3.8,3);
    \draw[-] (3,-3.8,0)-- (3,-3.8,3);
    \draw[-] (3,0,0)-- (3,0,3);
    
\end{tikzpicture}
\caption{The fluid and structure domains}
\vspace{1ex}
\end{minipage}
\hfill
\begin{minipage}[b]{0.5\linewidth}
    \centering
\tdplotsetmaincoords{70}{30}
\begin{tikzpicture}[tdplot_main_coords]
\coordinate (O) at (0,0,0);
\fill[blue!50,opacity=0.5] (1.5,-3.8,3) -- (1.5,0,3) -- (0,0,3) -- (0,-3.8,3) -- cycle;

\fill[blue!50,opacity=0.5] (1.5,-3.8,1.5) -- (1.5,0,1.5) -- (0,0,1.5) -- (0,-3.8,1.5) -- cycle;

\fill[blue!50,opacity=0.5] (1.5,-3.8,0) -- (1.5,0,0) -- (0,0,0) -- (0,-3.8,0) -- cycle;

\fill[magenta!50,opacity=0.5] (3,-3.8,3) -- (3,0,3) -- (1.5,0,3) -- (1.5,-3.8,3) -- cycle;

\fill[magenta!50,opacity=0.5] (3,-3.8,0) -- (3,0,0) -- (1.5,0,0) -- (1.5,-3.8,0) -- cycle;

    \draw[->] (O) --++ (3.5,0,0) node[below] {$x$};
    \draw[->] (O) --++ (0,-4.5,0) node[below] {$y$};
    \draw[->] (O) --++ (0,0,3.5) node[right] {$t$};
    \draw[dashed] (1.5,-3.8,0)-- (1.5,0,0);
    \draw[dashed] (1.5,0,0)-- (1.5,0,3);
    \node[] at (0.75,-1.7,0) {$ \textcolor{blue}{\Omega_f} $};
    \node[] at (2.4,-1.7,0) {$ \textcolor{red}{\Omega_s} $};

     \draw[-] (0,0,3)-- (0,-3.8,3);
    \draw[-] (1.5,0,3)-- (1.5,-3.8,3);
    \draw[-] (3,0,3)-- (3,-3.8,3);
        
    \draw[-] (1.5,-3.8,3)-- (1.5,-3.8,0);
    \draw[-] (1.5,-3.8,0)-- (3,-3.8,0);
    
    \draw[-] (1.5,-3.8,3)-- (3,-3.8,3);
    \draw[-] (1.5,0,3)-- (3,0,3);
    \draw[-] (0,0,3)-- (1.5,0,3);
    \draw[-] (0,-3.8,3)-- (1.5,-3.8,3);
    \draw[-] (0,-3.8,0)-- (1.5,-3.8,0);
    
    \draw[-] (3,-3.8,0)-- (3,0,0);

    \draw[-] (0,-3.8,0)-- (0,-3.8,3);
    \draw[-] (3,-3.8,0)-- (3,-3.8,3);
    \draw[-] (3,0,0)-- (3,0,3);
    
\end{tikzpicture}
\caption{Nonconforming time grids}
\vspace{1ex}
\end{minipage}    

\centering
\begin{minipage}[b]{0.5\linewidth}
\centering
\begin{tikzpicture}[scale=0.5]
\draw[thick, <-] (0,7)--(0,0);
\draw[thick, ->] (11,0)--(12,0);

\draw[thick, blue] (0,0)--(5,0);
\draw[thick, blue] (0,1.2)--(5,1.2);
\draw[thick, blue] (0,2.4)--(5,2.4);
\draw[thick, blue] (0,3.6)--(5,3.6);
\draw[thick, blue] (0,4.8)--(5,4.8);
\draw[thick, blue] (0,6)--(5,6);

\draw[thick, orange] (5,0)--(5,6);

\draw[thick, red] (5,0)--(11,0);
\draw[thick, red] (5,1.5)--(11,1.5);
\draw[thick, red] (5,3)--(11,3);
\draw[thick, red] (5,4.5)--(11,4.5);
\draw[thick, red] (5,6)--(11,6);

\draw[thick, -] (11,0)--(11,6);

\node [below] at (12,0){$ x $};
\node[left] at (0,7) {$ \large{t} $};
\node[left] at (0,6) {$ \large{T} $};
\node[left] at (0,0) {$ 0 $};

\draw[thick, <->] (-0.25,2.4)--(-0.25,3.6);
\node[left, blue] at (-0.25,3) {$ \Delta t_f^m $};

\draw[thick, <->] (11.25,3)--(11.25,4.5);
\node[right, red] at (11.25,3.75) {$ \Delta t_s^n $};

\node [below, blue] at (2.5,0){$ \Omega_f $};
\node [below, red] at (8,0){$ \Omega_s $};
\end{tikzpicture}\\
\caption{Projection of nonconforming time grids in two dimensions}
\label{fig:fig2}
\vspace{1ex}
\end{minipage}  
\end{figure}

We discretize the FSI system in time using an additional variable $\dot \beeta \in \bSigma$, representing $\partial_t \beeta$.
 Using the backward Euler method, the semi-discrete FSI system with Robin transmission conditions \eqref{R1} and  \eqref{R2} is given by: for $m=1,.....,M_f$
\begin{eqnarray} 
& &   \rho_f(\bu^{m}-\bu^{m-1}) + \Delta t_f^m  ( - 2\nu_f  \nabla \cdot D(\bu^{m}) + \nabla p^{m} I)\,=\,\bf_f^m  \quad \mbox{in } \Omega_f \,,  \label{fsi_stokes_dis_1}  \\
& &    \nabla \cdot \bu^{m}\,=\,0 \quad \mbox{in } \Omega_f \,,  \label{fsi_stokes_dis_2}
  \\
& &  \Delta t_f^m \left(\alpha_f \bu^m + \bsigma_f^m\bn_f\right) \,=
\int_{J_f^m}\Pi_{f,s}\left(\alpha_f \dot \beeta - \bsigma_s\bn_s\right) \, dt \; \mbox{on } \Gamma, \label{fsi_stokes_dis_3}
 \end{eqnarray}
  and for $n=1,.....,M_s$
 \begin{eqnarray}
& &\rho_s (\dot \beeta^n - \dot \beeta^{n-1}) - 2 \nu_s \Delta t_s^n \, \nabla \cdot D(\beeta^n)  - \lambda \Delta t_s^n \nabla (\nabla \cdot \beeta^n) =  \bbf_s^n  \quad \mbox{in } \Omega_s \,, \label{fsi_str_dis_1}\\
& &  (\beeta^n-\beeta^{n-1}) - {\Delta t_s^n} \dot \beeta^n = \b0   \quad \mbox{in } \Omega_s
\label{fsi_str_dis_2}\\
& & \Delta t_s^n \left(-\alpha_s\dot \beeta^{n} -\bsigma^n_s\bn_s\right)\,=\,\int_{J_s^n}\Pi_{s,f}\left(-\alpha_s \bu + \bsigma_f\bn_f\right)dt \quad \mbox{on } \Gamma, \label{fsi_str_dis_3}
\end{eqnarray}
where $(\bu,p)= (\bu^m,p^m)_{m=1}^{M_f}$ satisfies the boundary condition (\ref{fluidbd}) and the initial condition (\ref{fluidin}), and $(\beeta, \dot \beeta) = (\beeta^n, \dot \beeta^n)_{n=1}^{M_s}$ satisfies the boundary condition (\ref{porobd}) and the initial conditions (\ref{poroin1}), (\ref{poroin}). 

Next, we present the semi-discrete SWR algorithm and prove the convergence of the iterations.
Consider the following algorithm. In the $k$th iteration step, solve 
\begin{eqnarray} 
& &    \rho_f(\bu^{k,m}-\bu^{k,m-1}) + \Delta t_f^m  ( - 2\nu_f  \nabla \cdot D(\bu^{k,m}) + \nabla p^{k,m} I)\,=\,\bf_f^m  \quad \mbox{in } \Omega_f \,,  \label{fsi_dis_Stokes_1}  
\\
& &  \nabla \cdot \bu^{k,m}\,=\,0 \quad \mbox{in } \Omega_f \,,  \label{fsi_dis_stokes_2}
  \\
& &  \Delta t_f^m \left(\alpha_f \bu^{k,m} + \bsigma_f^{k,m}\bn_f\right) \,=
\int_{J_f^m}\Pi_{f,s}\left(\alpha_f \dot \beeta^{k-1} - \bsigma_s^{k-1} \bn_s\right) \, dt \; \mbox{on } \Gamma, \label{fsi_dis_stokes_3}
 \end{eqnarray}
 for $(\bu^{k,m},p^{k,m})$ satisfying \eqref{fluidbd} and \eqref{fluidin}, where $\bu^{k,0} = \bu^{0}$, $\bu^{k,m}:=\bu^{k}|_{J_f^m},p^{k,m}:=p^{k}|_{J_f^m}$ for $m=1,.....,M_f$,
 and
 \begin{eqnarray}
& &\rho_s (\dot \beeta^{k,n} - \dot \beeta^{k,n-1}) - 2 \nu_s \Delta t_s^n \, \nabla \cdot D(\beeta^{k,n})  - \lambda \Delta t_s^n \nabla (\nabla \cdot \beeta^{k,n}) =  \bbf_s^n  \quad \mbox{in } \Omega_s \,, \,\label{fsi_dis_str_1} \\ 
& &  (\beeta^{k,n}-\beeta^{k,n-1}) - {\Delta t_s^n} \dot \beeta^{k,n} = \b0   \quad \mbox{in } \Omega_s,
\label{fsi_dis_str_2}
\\
& & \Delta t_s^n \left(-\alpha_s \dot \beeta^{k,n} -\bsigma^{k,n}_s\bn_s\right)\,=\,\int_{J_s^n}\Pi_{s,f}\left(-\alpha_s \bu^{k-1} + \bsigma^{k-1}_f\bn_f\right)dt \quad \mbox{on } \Gamma \label{fsi_dis_str_3}
\end{eqnarray}
for $(\beeta^{k,n}, \dot\beeta^{k,n})$ satisfying \eqref{porobd}-\eqref{poroin}, where $\beeta^{k,0}=\beeta^{0}$, $\beeta^{k,n}:=\beeta^{k}|_{J_s^n}$\,
$\dot \beeta^{k,0}=\bar{\beeta}^{0}$, $\dot \beeta^{k,n}:=\dot \beeta^{k}|_{J_s^n}$ for $n=1,.....,M_s.$ 
In the next theorem 
  we show that the  weak solution to  \eqref{fsi_dis_Stokes_1}-\eqref{fsi_dis_str_3}
  converges to the weak solution of
 \eqref{fsi_stokes_dis_1}-\eqref{fsi_str_dis_2}  as $k\rightarrow \infty$.
\begin{remark}
The proposed method allows varying time step sizes across subdomains, with interface conditions on nonconforming time grids enforced through the $L^2$ projection. Being fully implicit in time, the method supports the use of large time step sizes while maintaining stability.
\end{remark}
 \begin{theorem}
Suppose the initial guess  $(\bu^0, p^0, \beeta^0,\dot \beeta^0)$ is chosen such that the Robin-Robin conditions \eqref{fsi_dis_stokes_3} and
\eqref{fsi_dis_str_3} are well-defined in $\textbf{L}^2({\Gamma})$. Further, assume $\bsigma_s \in \bH^1(\Omega_s)$ and the condition that $\alpha_f=\alpha_s$, then the weak formulation \eqref{fsi_dis_Stokes_1}-\eqref{fsi_dis_str_3} defines a unique sequence of iterates 
 $$(\bu^k, p^k, \beeta^k, \dot \beeta^k) \in P_0(\tau_f, \textbf{X}) \times  P_0(\tau_f,Q)  \times P_0(\tau_s, \bSigma)\times P_0(\tau_s, \bSigma)$$
 that converges to the weak solution of  \eqref{fsi_stokes_dis_1}-\eqref{fsi_str_dis_3}.
 \label{theo2discrete}
 \end{theorem}
 \begin{proof}
 Since the equations are linear, we can set $\bf_f=\bf_s=\bu^{0}=\beeta^0=\dot \beeta^0=\b0$, and proceed to derive energy estimates following the proof of Theorem \ref{theo1}. The structure of the proof is similar to \Cref{theo1}, except that we do not use the Gronwall lemma due to the presence of projection operators.  
We multiply the equations 
\eqref{fsi_dis_Stokes_1} and \eqref{fsi_dis_stokes_2}
by $\bu^{k,m}$ and 
 $p^{k,m}$, respectively,  use \eqref{sq1}, and add the resulting equations together to obtain
 \begin{eqnarray}
 & & \rho_f (\bu^{k,m},\bu^{k,m})_{\Omega_f}-\rho_f(\bu^{k,m-1},\bu^{k,m})_{\Omega_f}
 +2\nu_{f}\Delta t_f^m \norm{D(\bu^{k,m})}_{\Omega_f}^2 \nn \\ && \hspace{.5in}
 +\frac{\Delta t_f^m}{2(\alpha_f + \alpha_s)}\norm{\bsigma_f^{k,m}\bn_f - \alpha_s \bu^{k,m}}^2_{\Gamma} \nn \\ && \hspace{.5in}
 \leq  \frac{\Delta t_f^m}{2(\alpha_f + \alpha_s)}\norm{\bsigma_f^{k,m}\bn_f + \alpha_f\bu^{k,m}}^2_{\Gamma} +\frac{\Delta t_f^m}{2}(\alpha_s - \alpha_f)\norm{\bu^{k,m}}^2_{\Gamma}.
\end{eqnarray}
By using Cauchy-Schwarz inequality and $\frac{1}{2}(a^2-b^2)\leq a^2-ab$, we can obtain
 \begin{eqnarray}\label{fsi_f}
 & & \frac{\rho_f}{2}\left(\|\bu^{k,m}\|_{\Omega_f}^2-\|\bu^{k,m-1}\|_{\Omega_f}^2\right)
 +2\nu_{f} \int_{J_f^m} \norm{D(\bu^{k})}_{\Omega_f}^2 dt \nn \\ & & \hspace{.5in}
 +\frac{1}{2(\alpha_f + \alpha_s)}\int_{J_f^m} \norm{\bsigma_f^{k}\bn_f - \alpha_s \bu^{k}}^2_{\Gamma} dt  \nn \\ & & \hspace{.5in}
 \leq  \frac{1}{2(\alpha_f + \alpha_s)}\int_{J_f^m} \norm{\bsigma_f^{k}\bn_f + \alpha_f\bu^{k}}^2_{\Gamma}dt +\frac{(\alpha_s -\alpha_f)}{2}\int_{J_f^m}\norm{\bu^{k}}^2_{\Gamma}dt.
\end{eqnarray}
Next, multiply  \eqref{fsi_dis_str_1} by $\dot \beeta^{k,n}$, integrate over $\Omega_s$ and  use \eqref{fsi_dis_str_2} and \eqref{sq21} to have
\begin{eqnarray}\label{fsi_s}
& & \frac{\rho_s}{2}\left(\|\dot \beeta^{k,n}\|_{\Omega_s}^2-\|\dot \beeta^{k,n-1}\|_{\Omega_s}^2\right)
+ \nu_s \left(\norm{D(\beeta^{k,n})}^2_{\Omega_s}-\norm{D(\beeta^{k,n-1})}^2_{\Omega_s}\right) \nn \\ & & \hspace{.5in}
+ \frac{\lambda}{2} \left(\norm{\nabla \cdot \beeta^{k,n}}^2_{\Omega_s}-\norm{\nabla \cdot \beeta^{k,n-1}}^2_{\Omega_s}\right)
+ \frac{1}{2(\alpha_f + \alpha_s)} \int_{J_s^n}\norm{-\bsigma_s^{k}\bn_s + \alpha_f \dot \beeta^{k}}^2_{\Gamma}dt \nn \\ & & \hspace{.5in}
 \leq  \frac{1}{2(\alpha_f + \alpha_s)} \int_{J_s^n}\norm{-\bsigma_s^{k}\bn_s - \alpha_s \dot \beeta^{k}}^2_{\Gamma}dt
 -\frac{(\alpha_s-\alpha_f)}{2} \int_{J_s^n} \norm{\dot \beeta^{k}}^2_{\Gamma}dt.
\end{eqnarray}

To eliminate the last terms of equations \eqref{fsi_f} and \eqref{fsi_s}, we assume that $\alpha_f=\alpha_s$, since we cannot use Gronwall's lemma 
as for the continuous case due to the global-in-time projection $\Pi_{f,s}$ and $\Pi_{s,f}$ \cite{HJJ13}.  By summing these equations over the subintervals in $(0,t_f^m]$ and $(0,t_s^n]$, respectively, we obtain
\begin{eqnarray}\label{fsi_f1}
 \frac{\rho_f}{2}\|\bu^{k,m}\|_{\Omega_f}^2
 +2\nu_{f} \int_{0}^{t_f^m} \norm{D(\bu^{k})}_{\Omega_f}^2 dt +\frac{1}{2(\alpha_f + \alpha_s)}\int_{0}^{t_f^m} \norm{\bsigma_f^{k}\bn_f - \alpha_s \bu^{k}}^2_{\Gamma} dt  \nn \\ \hspace{.5in}
 \leq  \frac{1}{2(\alpha_f + \alpha_s)}\int_{0}^{t_f^m} \norm{\bsigma_f^{k}\bn_f + \alpha_f\bu^{k}}^2_{\Gamma}dt,
\end{eqnarray}
and
\begin{eqnarray}\label{fsi_s1}
& & \frac{\rho_s}{2}\|\dot \beeta^{k,n}\|_{\Omega_s}^2 + \nu_s \norm{D(\beeta^{k,n})}^2_{\Omega_s} + \frac{\lambda}{2} \norm{\nabla \cdot \beeta^{k,n}}^2_{\Omega_s}  \nn \\ & & \hspace{.5in}
+ \frac{1}{2(\alpha_f + \alpha_s)} \int_{0}^{t_s^n}\norm{-\bsigma_s^{k}\bn_s + \alpha_f \dot \beeta^{k}}^2_{\Gamma}dt \nn \\ & & \hspace{.5in}
 \leq  \frac{1}{2(\alpha_f + \alpha_s)} \int_{0}^{t_s^n}\norm{-\bsigma_s^{k}\bn_s - \alpha_s \dot \beeta^{k}}^2_{\Gamma}dt.
\end{eqnarray}
We add \eqref{fsi_f1} and \eqref{fsi_s1}, apply the Robin conditions \eqref{fsi_dis_stokes_3} and \eqref{fsi_dis_str_3} and set $t_f^m=t_s^n=T$ to obtain the following.
\begin{eqnarray}\label{fsi_fs}
 & & \frac{\rho_f}{2}\|\bu^{k,M_f}\|_{\Omega_f}^2
 +2\nu_{f} \int_{0}^{T} \norm{D(\bu^{k})}_{\Omega_f}^2 dt \nn \\ & & \hspace{.5in}
  \frac{\rho_s}{2}\|\dot \beeta^{k,M_s}\|_{\Omega_s}^2 + \nu_s \norm{D(\beeta^{k,M_s})}^2_{\Omega_s} + \frac{\lambda}{2} \norm{\nabla \cdot \beeta^{k,M_s}}^2_{\Omega_s}
\nn \\ & & \hspace{.5in}
+\frac{1}{2(\alpha_f + \alpha_s)}\int_{0}^{T} \norm{\bsigma_f^{k}\bn_f - \alpha_s \bu^{k}}^2_{\Gamma} dt  + \frac{1}{2(\alpha_f + \alpha_s)} \int_{0}^{T}\norm{-\bsigma_s^{k}\bn_s + \alpha_f \dot \beeta^{k}}^2_{\Gamma}dt \nn \\ & & \hspace{.5in}
 \leq  \frac{1}{2(\alpha_f + \alpha_s)}\int_{0}^{T} \norm{\bsigma_f^{k}\bn_f + \alpha_f\bu^{k}}^2_{\Gamma}dt+\frac{1}{2(\alpha_f + \alpha_s)} \int_{0}^{T}\norm{-\bsigma_s^{k}\bn_s - \alpha_s \dot \beeta^{k}}^2_{\Gamma}dt \nn \\ & & \hspace{.5in}
 \leq \frac{1}{2(\alpha_f + \alpha_s)}\int_{0}^{T} \norm{\Pi_{s,f}(\bsigma_f^{k-1}\bn_f - \alpha_s \bu^{k-1})}^2_{\Gamma} dt \nn \\ & & \hspace{.5in}
 + \frac{1}{2(\alpha_f + \alpha_s)} \int_{0}^{T}\norm{\Pi_{f,s}(-\bsigma_s^{k-1}\bn_s + \alpha_f \dot \beeta^{k-1})}^2_{\Gamma}dt \nn \\ & & \hspace{.5in}
\leq \frac{1}{2(\alpha_f + \alpha_s)}\int_{0}^{T} \norm{\bsigma_f^{k-1}\bn_f - \alpha_s \bu^{k-1}}^2_{\Gamma} dt \nn \\ & & \hspace{.5in}
 + \frac{1}{2(\alpha_f + \alpha_s)} \int_{0}^{T}\norm{-\bsigma_s^{k-1}\bn_s + \alpha_f \dot \beeta^{k-1}}^2_{\Gamma}dt \,.
\end{eqnarray}
Then, for all $k>0$
\begin{eqnarray}
 & & \frac{\rho_f}{2}\|\bu^{k,M_f}\|_{\Omega_f}^2
 +2\nu_{f} \int_{0}^{T} \norm{D(\bu^{k})}_{\Omega_f}^2 dt +\frac{\rho_s}{2}\|\dot \beeta^{k,M_s}\|_{\Omega_s}^2 + \nu_s \norm{D(\beeta^{k,M_s})}^2_{\Omega_s} \nn \\ & & \hspace{.5in}
 + \frac{\lambda}{2} \norm{\nabla \cdot \beeta^{k,M_s}}^2_{\Omega_s}
+ B^k \leq B^{k-1},
\end{eqnarray}
where 
$$B^k =\frac{1}{2(\alpha_f + \alpha_s)}\int_{0}^{T} \norm{\bsigma_f^{k}\bn_f - \alpha_s \bu^{k}}^2_{\Gamma} dt  + \frac{1}{2(\alpha_f + \alpha_s)} \int_{0}^{T}\norm{-\bsigma_s^{k}\bn_s + \alpha_f \dot \beeta^{k}}^2_{\Gamma}dt \, .$$
 By summing over the iterates $k$, we conclude that $\|\bu^{k,M_f}\|_{\Omega_f}, \int_{0}^{T} \|D(\bu^k)\|_{\Omega_f}^2  dt, \|D(\beeta^{k,M_s})\|_{\Omega_s}  $,  $\norm{\nabla \cdot \beeta^{k,M_s}}_{\Omega_s}$ and  $ \|\dot \beeta^{k,M_s}\|_{\Omega_s}$ converge to $0$ as $k\rightarrow \infty$. This implies $\int_{0}^{t_f^m} \|D(\bu^k)\|_{\Omega_f}^2  dt$ converges to 0 as $k \rightarrow \infty$ for all $m=1,2, \cdot \cdot \cdot, M_f$, and also $\|D(\bu^{k,m})\|_{\Omega_f}$ converges $0$ for all $m=1,2, \cdot \cdot \cdot, M_f$,
as $\bu^k \in P_0(\tau_f,\bX)$.
  Now, using Korn's inequality,  $\|\bu^{k,m}\|_{\Omega_f} \leq C_{PF1} \|D(\bu^{k,m})\|_{\Omega_f}$ and $\|\beeta^{k,M_s}\|_{\Omega_s} \leq C_{PF2} \|D(\beeta^{k,M_s})\|_{\Omega_s}$ for some constants $C_{PF1}, C_{PF2} > 0$, which implies $\|\bu^{k,m}\|_{\Omega_f}$ converges to 0 for all $m=1,2, \cdot \cdot \cdot, M_f$, and $\|\beeta^{k,M_s}\|_{\Omega_s}$ also converges to $0$ as $k\rightarrow \infty$.

Next, we show the convergence of 
$\beeta^{k,n}$ $\dot \beeta^{k,n}$ and $D(\beeta^{k,n})$ 
in the $L^2$ norm for all $n=1,2, \cdot \cdot \cdot, M_s$.
We multiply \eqref{fsi_dis_str_1} by $\dot \beeta^{k,n-1}$, integrate over $\Omega_s$, and use Cauchy-Schwarz inequality and the Trace theorem to obtain 
\begin{eqnarray}
\|\dot \beeta^{k,n-1}\|_{\Omega_s}^2& \leq &   C_{s1} \left(\|\dot \beeta^{k,n}\|_{\Omega_s}\|\dot \beeta^{k,n-1}\|_{\Omega_s}+ \|D(\beeta^{k,n})\|_{\Omega_s} \|D(\dot \beeta^{k,n-1}\|_{\Omega_s} \right. \nn \\
&& \left. + \|\bsigma_s^{k,n}\bn_s\|_{1,\Omega_s}^{1/2}\|\bsigma_s^{k,n}\bn_s\|_{\Omega_s}^{1/2}\|\dot \beeta^{k,n-1}\|^{1/2}_{1,\Omega_s}
\|\dot \beeta^{k,n-1}\|^{1/2}_{\Omega_s}\right),
\label{s1eq-1}
\end{eqnarray} 
for some constant $C_{s1} > 0$.
Similarly, from \eqref{fsi_dis_str_2}, for some constant $C_{s2}>0$,
\begin{equation}\label{s2eq-1}
    \|\beeta^{k,n-1}\|_{\Omega_s} \leq C_{s2}(\|\beeta^{k,n}\|_{\Omega_s}+\|\dot \beeta^{k,n}\|_{\Omega_s})
    \,.
\end{equation}
For $n=M_s$, $\|\bsigma_s^{k,n}\bn_s\|_{\Omega_s}$ in \eqref{s1eq-1} converges to $0$ as 
$k \rightarrow \infty$ by its definition and the convergence of $\|D(\beeta^{k,M_s})\|_{\Omega_s}$, hence the last term in \eqref{s1eq-1} converges to $0$ if $\bsigma_s^{k,n}\bn_s$ has $H^1$ regularity.  
Then,
\eqref{s1eq-1} and \eqref{s2eq-1}, together with the fact that $\|\dot \beeta^{k,M_s}\|_{\Omega_s}$ and $\|\beeta^{k,M_s}\|_{\Omega_s}$ converge to 0,  imply $\|\dot \beeta^{k,M_s-1}\|_{\Omega_s}$ and $\|\beeta^{k,M_s-1}\|_{\Omega_s}$ converge to 0 as $k \rightarrow \infty$.
Now, multiplying \eqref{fsi_dis_str_1} by $\beeta^{k,n}$, integrating over $\Omega_s$ and using Cauchy-Schwarz inequality, \eqref{fsi_dis_str_3} and the Trace theorem, 
\begin{eqnarray}
& & \|D(\beeta^{k,n})\|_{\Omega_s}^2 + \|\nabla \cdot \beeta^{k,n}\|_{\Omega_s}^2\,\leq \, C_{s3}\left(\|\dot \beeta^{k,n}\|_{\Omega_s}\|\beeta^{k,n}\|_{\Omega_s} + \|\dot \beeta^{k,n-1}\|_{\Omega_s}\|\beeta^{k,n}\|_{\Omega_s} \right. \nn \\
& &  \left.  \hspace{.8in} +\|\Pi_{s,f}|_{J^n_s} (\bu^{k-1} + \bsigma_f^{k-1}\bn_f) \|_{\Gamma} \|\beeta^{k,n}\|^{1/2}_{1,\Omega_s}
\|\beeta^{k,n}\|^{1/2}_{\Omega_s} \right)
\label{s3eq-1}
\end{eqnarray} 
for some $C_{s3}>0$.
Since $\bg_f \in \textbf{L}^2(\Gamma)$, the convergence of $\|\beeta^{k, M_s-1}\|_{\Omega_s}$ to $0$ implies that $\|D(\beeta^{k,M_s-1})\|_{\Omega_s}$ converges to 0. Therefore, in this way, we can show that  $\|\beeta^{k,n}\|_{\Omega_s}$,  $\|\dot \beeta^{k,n}\|_{\Omega_s}$ and $\|D(\beeta^{k,n})\|_{\Omega_s}$ converge to 0 as $k \rightarrow \infty$ for all $n=1,2,\ldots,M_s$. 

To establish the convergence of $p^{k,m}$, we follow a similar approach to the continuous case, and obtain the following result for some $\beta >0$:
\begin{equation*}
     \beta \|p^{k,m}\|_{\Omega_f} \, \leq \, ( 1+C_{*})( \,  \|D(\bu^{k,m})\|_{\Omega_f}+ \|\Pi_{f,s}\dot \beeta^{k-1,m}\|_{\Gamma}
+ \|\Pi_{f,s}\bsigma_s^{k-1,m}\bn_s\|_{\Gamma}  \, ) \, .
 \end{equation*}
Squaring both sides and integrating over the interval $(0,t_f^m]$, we get, for all $m$
 \begin{eqnarray}\label{abc}
     \frac{\beta^2}{3( 1+C_{*})} \int_0^{t_f^m} \|p^{k}\|^2_{\Omega_f}  \; dt \, & & \leq \, \int_0^{t_f^m} \, \|D(\bu^k)\|^2_{\Omega_f} ds + \int_0^{T}\|\dot \beeta^{k-1}\|_{1,\Omega_s}\|\dot \beeta^{k-1}\|_{\Omega_s} ds \nn \\
&& + \int_0^{T} \|\bsigma_s^{k-1}\|_{1,\Omega_s}\|\bsigma_s^{k-1}\|_{\Omega_s}  \, \; ds\, .
 \end{eqnarray}
The last term in \eqref{abc} converges to $0$ by the regularity assumption for  $\bsigma_s$ and the convergence of $\|D(\beeta^{k,n})\|_{\Omega_s}$ for all $n=1,2,\ldots,M_s$. 
Then, the convergence of $\int_0^{t_f^m} \|p^{k}\|_{\Omega_f}^2 dt$ 
to $0$ as $k \rightarrow \infty$ follows from the convergence of  $\|D(\bu^{k,m})\|^2_{\Omega_f}$ and 
$\|\dot \beeta^{k,n}\|_{\Omega_s}$ for all $m$ and $n$. Finally, 
we have that $\|p^{k,m}\|_{\Omega_f}$ converges to $0$ for all $ m=1,2,\ldots,M_f$, 
as $p^k \in P_0(\tau_f,Q)$.
\end{proof}

\section{Numerical Examples}\label{sec:numericalresults}
We provide two numerical examples to demonstrate the effectiveness of our proposed methods.  The first example is a manufactured problem with a known solution, which we use to assess the accuracy and efficiency of the methods. The second example is a benchmark problem from the field of hemodynamics that has been previously considered in \cite{R14, R10}.
For both examples, we use GMRES
to achieve fast convergence when solving the interface problem \eqref{sp1} or \eqref{ifprob_weak}.   
The SWR algorithm analyzed in Section \ref{sec:semidiscrete} is a Jacobi-type iterative method, thus we expect fast convergence to be achieved by GMRES. No preconditioner is considered for the interface problems, and all computations are performed using freeFEM++ \cite{ff} in a sequential setting.

\subsection{Test 1 : With a Known Analytical Solution}
Consider an example with a known exact solution, where the fluid subdomain is $\Omega_f = (0, 1) \times (0, 1)$ and the structure subdomain is $\Omega_s = (0, 1) \times (1, 2)$. The interface between the two subdomains is given by $\Gamma = \{(x, y): 0 < x < 1, y = 1\}$. The chosen exact solution is 
\begin{eqnarray*}
&& \bu  = \begin{pmatrix}
\cos{(x+t)}\sin{(y+t)}+\sin{(x+t)}\cos{(y+t)} \\
-\sin{(x+t)}\cos{(y+t)}-\cos{(x+t)}\sin{(y+t)}
\end{pmatrix} \\
&&  p  = 2\nu_f (\sin{(x+t)}\sin{(y+t)}-\cos{(x+t)}\cos{(y+t)}) +  2\nu_s \cos{(x+t)}\sin{(y+t)} \\
&& \beeta  = \begin{pmatrix}
\sin{(x+t)}\sin{(y+t)} \\
\cos{(x+t)}\cos{(y+t)}
\end{pmatrix}. 
\end{eqnarray*}
The constants, $\rho_s$, $\rho_f$, $\nu_s$, $\nu_f$, and $\lambda$, are set to unity. 
 For the Robin conditions \eqref{R1} and \eqref{R2} we choose  $\alpha_f=1$ and $\alpha_s=100$, and 
the tolerance
for GMRES is set to  $\epsilon = 10^{-7}$. 

We test the convergence of both methods in space using inf-sup stable Taylor-Hood elements $(\textbf{P}2, P1)$ for the fluid subproblem and $\textbf{P}2$ elements for the structure subproblem, along with nonconforming time grids. We then repeat the test using inf-sup stable MINI elements $(\textbf{P}1 + bubble, P1)$ \cite{miniexp} for the fluid and $\textbf{P}1$ element for the structure subproblem. In all numerical tests, subdomains are discretized with matching grid points on the interface. 
The unknown Lagrange multipliers are approximated on the common grid points along the interface using the same polynomial degrees for the displacement approximation.  

The errors at the final time $T=0.0025$ are presented in 
Tables  \ref{spspaceT}-\ref{SWRspace} for all variables,  with expected convergence rates shown. For Taylor-Hood elements we expect convergence of $O(h^2)$ and for MINI elements we expect convergence if $O(h)$ for velocity in $H_1$ norm and for pressure $O(h^2)$ in $L_2$ norm. To test convergence with respect to different time steps, we use $\Delta t_{\text{coarse}}$ to denote the coarse time step size and set the fine time step size to be $\Delta t_{\text{fine}}= \Delta t_{\text{coarse}}/2$.   Numerical tests are performed using three different types of time grids:
 \begin{enumerate}
     \item Coarse conforming time grids: $\Delta t_f = \Delta t_s = \Delta t_{\text{coarse}}$,
     \item Fine conforming time grids: $\Delta t_f = \Delta t_s = \Delta t_{\text{fine}}$,
     \item Nonconforming time grids: $\Delta t_f = \Delta t_{\text{coarse}}$ and $\Delta t_s = \Delta t_{\text{fine}}$.
 \end{enumerate}
First, the Steklov-Poincar\'{e} interface problem \eqref{sp1} is solved using
 Taylor-Hood and \textbf{P}2 elements with $h=\frac{1}{32}$  and the three types of time grids given above, with  $\Delta t_{\text{coarse}} \in \{0.2, 0.1, 0.05, 0.0025 \}$. Then the same test is repeated using MINI and \textbf{P}1 elements with $h=\frac{1}{64}$ and $\Delta t_{\text{coarse}} \in \{0.4, 0.2, 0.1, 0.05 \}$.
 %
 Figures  \ref{sptime1} and \ref{sptime1-MINI} demonstrate the first-order convergence of solutions, showing the errors at $T=0.2$ and $T=0.4$, respectively.
 For the nonconforming time grids, 
 $\Delta t_f = \Delta t_{\text{coarse}}$ and $\Delta t_s = \Delta t_{\text{fine}}$. Thus, as expected, 
 the fluid velocity and pressure errors for the nonconforming time grids are close to the errors of the conforming coarse grids, while the displacement errors are between the errors of the conforming fine and coarse grids. 
 We also solve the interface problem 
 \eqref{ifprob_weak} using 
 Taylor-Hood and \textbf{P}2 elements with the same condition as \eqref{sp1}.
 Figure \ref{sptime1-Robin} shows a result similar to the result obtained by the  Steklov-Poincar\'{e} method in Figure \ref{sptime1}. 
 
Next, we compare the computer running times for both methods using conforming and nonconforming time grids on a fixed mesh. Computer running times for  both methods are presented in Table \ref{spcpu1hk} and Table \ref{spcpu1hk2}. The tables show that the computer running times for the nonconforming cases are close to the conforming coarse cases than the conforming fine cases for both methods, demonstrating the efficiency of the proposed methods. 
 We also examine the convergence behavior of GMRES by various Robin parameters, $\alpha_f$ and $\alpha_s$.
 Table \ref{tabiter} presents the number of iterations for various $\alpha_s$ values when $\alpha_f$ is fixed to 1. The table indicates that 
 a higher $\alpha_s$ value yields faster convergence of GMRES iterations. However, in an additional test, we 
observe that the convergence of GMRES is not much affected by  $\alpha_f$ values.

\begin{table}[ht]
\centering 
\begin{tabular}{l c c ccc} 
\hline\hline 
 $h$ & & 1/4 & 1/8 & 1/16 & 1/32
\\ [0.5ex]
\hline 
 &$L^2$ error & 2.08e-04 & 2.17e-05 [3.26] & 2.75e-06 [2.98] & 3.77e-07 [2.87] \\[-1ex]
\raisebox{1.5ex}{$\textbf{u}$} &$H^1$ error
& 5.90e-03 & 1.23e-03 [2.26] & 3.02e-04 [2.02] & 7.47e-05 [2.02] \\[1ex]
\hline
$p$ &$L^2$ error & 5.21e-03 & 1.09e-03 [2.26] & 2.54e-04 [2.10] & 6.56e-05 [1.95] \\[1ex]
\hline
 &$L^2$ error & 2.21e-04 & 2.31e-05 [3.26] & 2.38e-06 [3.28] & 3.11e-07 [2.94] \\[-1ex]
\raisebox{1.5ex}{$\beeta$} &$H^1$ error
& 6.50e-03 & 1.39e-03 [2.22] & 3.00e-04 [2.21] & 7.52e-05 [2.00] \\[1ex]
\hline
\end{tabular}
\caption{Errors by the Steklov-Poincar\'{e} method using Taylor-Hood and \textbf{P}2 elements, $\Delta t_f=0.000025$ and $\Delta t_s=0.000050$.} 
\label{spspaceT}
\end{table}    
 
\begin{table}[ht]
\centering 
\begin{tabular}{l c c ccc} 
\hline\hline 
 $h$ & & 1/4 & 1/8 & 1/16 & 1/32
\\ [0.5ex]
\hline 
 &$L^2$ error & 9.03e-03 & 2.18e-03 [2.05] & 4.79e-04 [2.19] & 1.23e-04 [1.97] \\[-1ex]
\raisebox{1.5ex}{$\textbf{u}$} &$H^1$ error
& 1.96e-01 & 8.12e-02 [1.27] & 3.58e-02 [1.18] & 1.80e-02[1.00] \\[1ex]
\hline
$p$ &$L^2$ error & 5.73e-01 & 1.43e-01 [2.00] & 3.59e-02 [2.00] & 1.06e-02 [1.76] \\[1ex]
\hline
 &$L^2$ error & 6.37e-03 & 1.34e-03 [2.24] & 3.35e-04 [2.00] & 8.43e-05 [1.99] \\[-1ex]
\raisebox{1.5ex}{$\beeta$} &$H^1$ error
& 8.25e-02 & 3.71e-02 [1.15] & 1.86e-02 [1.00] & 9.38e-03 [0.99] \\[1ex]
\hline
\end{tabular}
\caption{Errors by the Steklov-Poincar\'{e} method using MINI and \textbf{P}1 elements, $\Delta t_f=0.000025$ and $\Delta t_s=0.000050$.} 
\label{tab:5MRT2}
\end{table}    

\begin{table}[ht]
\centering 
\begin{tabular}{l c c ccc} 
\hline\hline 
 $h$ & & 1/4 & 1/8 & 1/16 & 1/32
\\ [0.5ex]
\hline 
 &$L^2$ error & 2.10e-04 & 2.09e-05 [3.34] & 2.56e-06 [3.03] & 3.39e-07 [2.92] \\[-1ex]
\raisebox{1.5ex}{$\textbf{u}$} &$H^1$ error
& 5.95e-03 & 1.23e-03 [2.27] & 3.04e-04 [2.03] & 7.43e-05 [2.02] \\[1ex]
\hline
$p$ &$L^2$ error & 5.40e-03 & 1.02e-03 [2.40] & 2.46e-04 [2.05] & 6.45e-05 [1.93] \\[1ex]
\hline
 &$L^2$ error & 2.21e-04 & 2.31e-05 [3.26] & 2.38e-06 [3.28] & 3.10e-07 [2.94] \\[-1ex]
\raisebox{1.5ex}{$\beeta$} &$H^1$ error
& 6.50e-03 & 1.39e-03 [2.22] & 3.00e-04 [2.21] & 7.52e-05 [2.00] \\[1ex]
\hline
\end{tabular}
\caption{Errors by the Robin method for $(\alpha_f, \alpha_s) = (1, 100)$ using Taylor-Hood and \textbf{P}2 elements, $\Delta t_f=0.000050$, $\Delta t_s=0.000025$.}
\label{SWRspace}
\end{table}

\begin{figure}[!htb] 
    \begin{minipage}[b]{0.5\linewidth}
    \centering
    \includegraphics[scale=0.3]{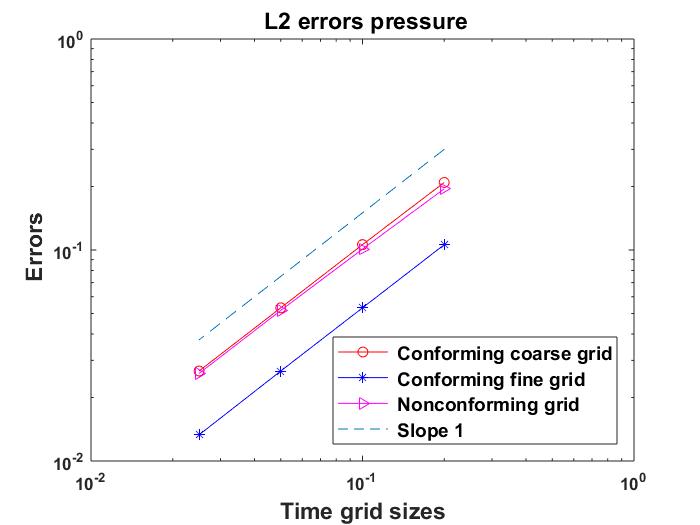} 
    \vspace{1ex}
  \end{minipage}
  \begin{minipage}[b]{0.5\linewidth}
    \centering
    \includegraphics[scale=0.3]{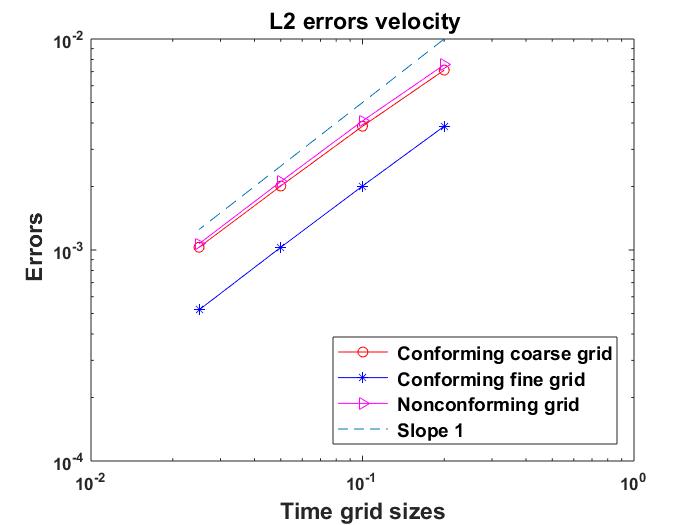} 
    \vspace{1ex}
  \end{minipage} 
  \begin{minipage}[b]{0.5\linewidth}
    \centering
    \includegraphics[scale=0.3]{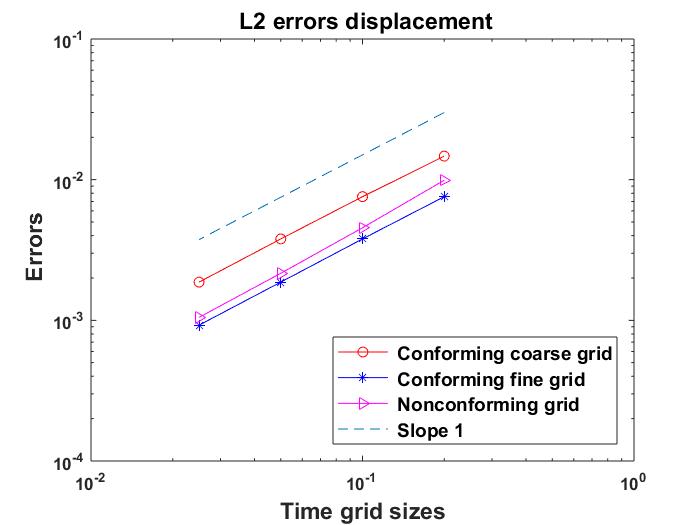} 
    \vspace{1ex}
  \end{minipage} 
    \begin{minipage}[b]{0.5\linewidth}
    \centering
    \includegraphics[scale=0.3]{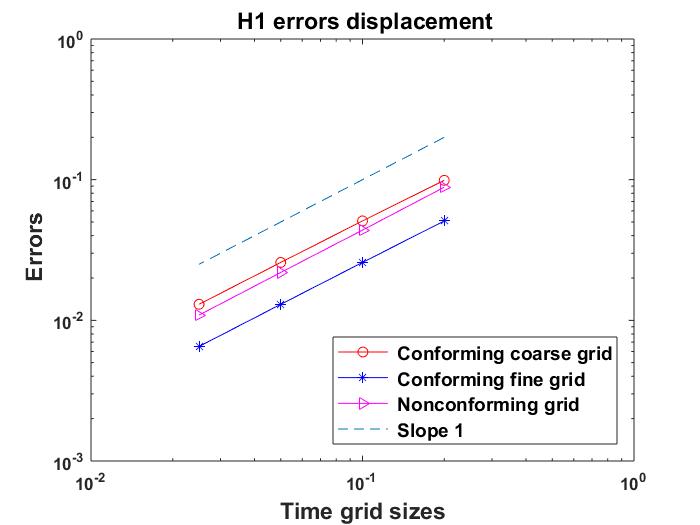} 
    \vspace{1ex}
  \end{minipage} 
  \caption{ Errors at $T=0.2$ by the Steklov-Poincar\'{e} method using Taylor-Hood and \textbf{P}2 elements with $h=1/32$.}
  \label{sptime1}
\end{figure}

\begin{figure}[!htb] 
    \begin{minipage}[b]{0.5\linewidth}
    \centering
    \includegraphics[scale=0.3]{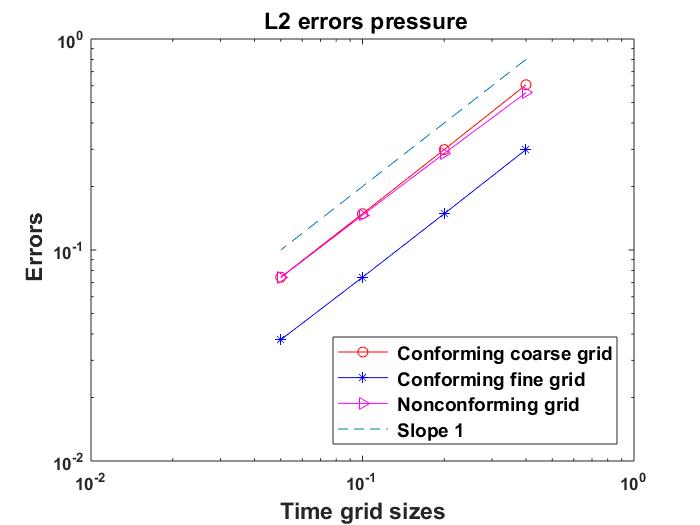} 
    \vspace{1ex}
  \end{minipage}
  \begin{minipage}[b]{0.5\linewidth}
    \centering
    \includegraphics[scale=0.3]{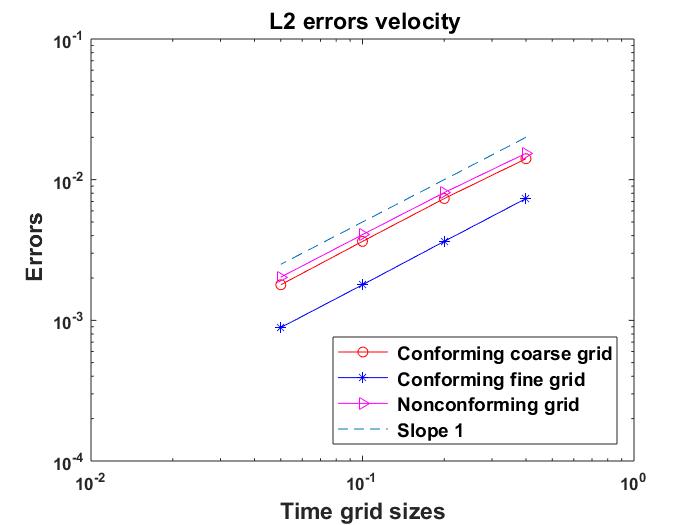} 
    \vspace{1ex}
  \end{minipage} 
  \begin{minipage}[b]{0.5\linewidth}
    \centering
    \includegraphics[scale=0.3]{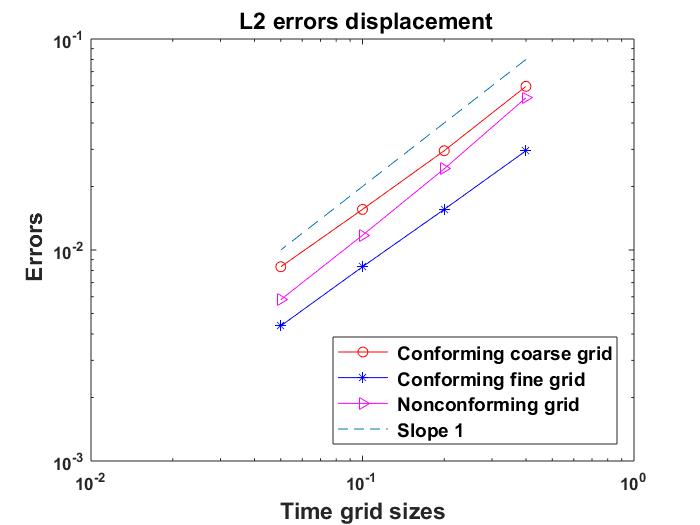} 
    \vspace{1ex}
  \end{minipage} 
    \begin{minipage}[b]{0.5\linewidth}
    \centering
    \includegraphics[scale=0.3]{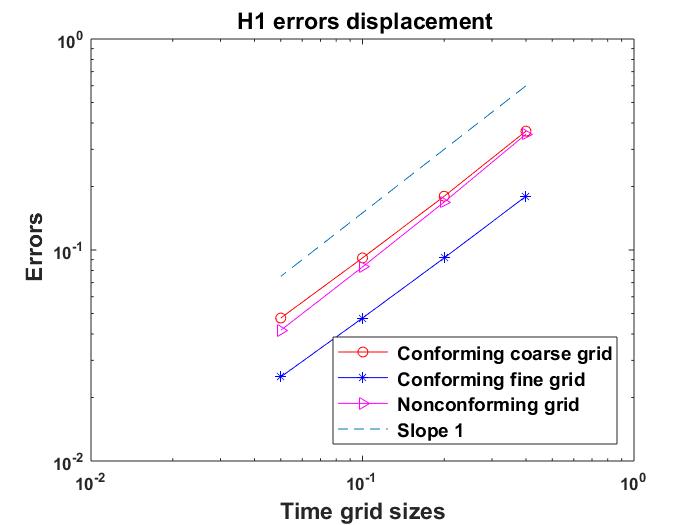} 
    \vspace{1ex}
  \end{minipage} 
  \caption{Errors at $T=0.4$ by the Steklov-Poincar\'{e} method using MINI and \textbf{P}1 elements with  $h=1/64$.}
   \label{sptime1-MINI}
\end{figure}

\begin{figure}[!htbp] 
    \begin{minipage}[b]{0.5\linewidth}
    \centering
    \includegraphics[scale=0.3]{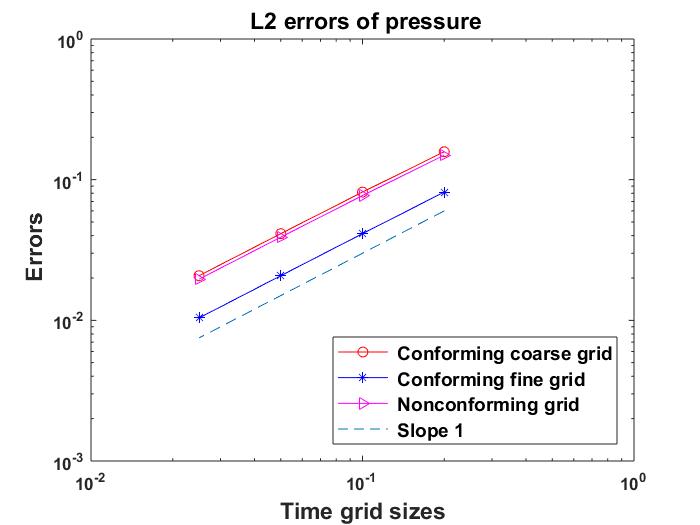} 
    \vspace{1ex}
  \end{minipage}
  \begin{minipage}[b]{0.5\linewidth}
    \centering
    \includegraphics[scale=0.3]{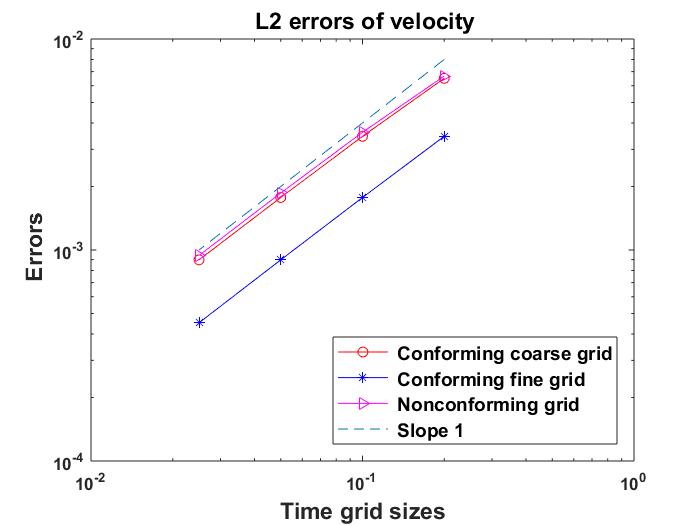} 
    \vspace{1ex}
  \end{minipage} 
  \begin{minipage}[b]{0.5\linewidth}
    \centering
    \includegraphics[scale=0.3]{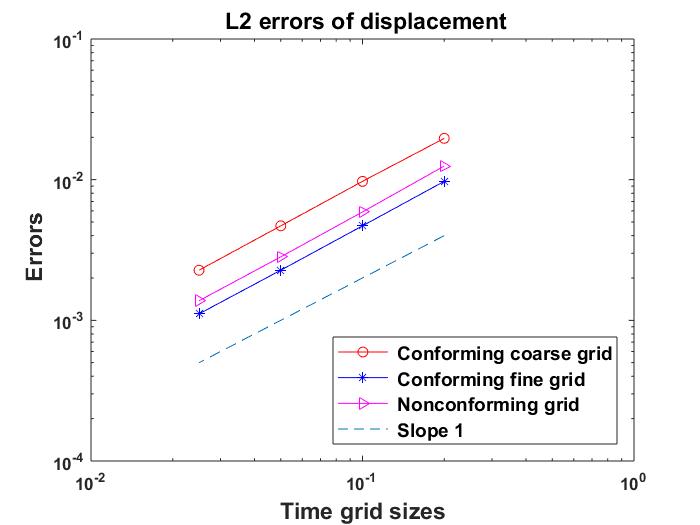} 
    \vspace{1ex}
  \end{minipage} 
    \begin{minipage}[b]{0.5\linewidth}
    \centering
    \includegraphics[scale=0.3]{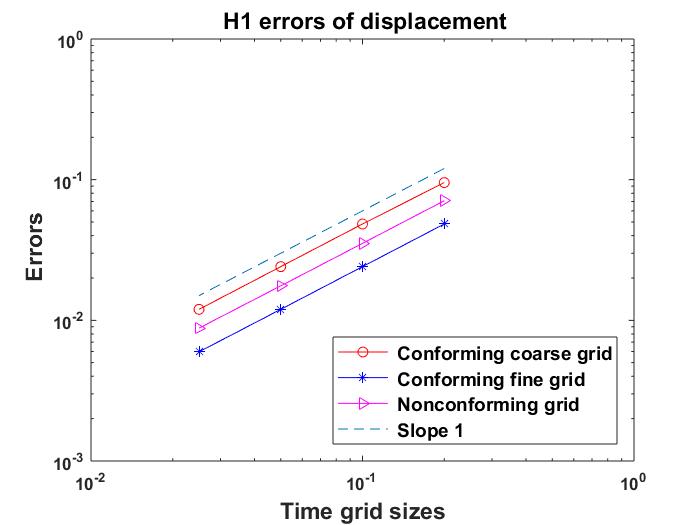} 
    \vspace{1ex}
  \end{minipage} 
  \caption{Errors at $T=0.2$ by the Robin method for $(\alpha_f, \alpha_s) = (1, 100)$ using  Taylor-Hood and \textbf{P}2 elements with $h=1/32$.}
   \label{sptime1-Robin}
\end{figure}


\begin{table}[!htbp] 
  \centering
  \begin{tabular}{c ccccc}
  \toprule
   {\multirow{4}{*}{$\Delta t$}} & \multicolumn{2}{c}{Steklov-Poincar\'{e} Method}& \multicolumn{2}{c}{Robin Method ($\alpha_f=1, \alpha_s=100$)} \\[2ex]
  \cmidrule{2-3}
  \cmidrule{4-5}
    & Conforming & Nonconforming & Conforming & Nonconforming\\[2ex]
    \midrule
    0.2 & 335 & {\multirow{4}{*}{376}} & 336 & {\multirow{4}{*}{338}}  \\[5ex]
    0.1 & 651 & {\multirow{4}{*}{802}} & 657 & {\multirow{4}{*}{676}}\\[5ex]
    0.05 & 1324 & {\multirow{4}{*}{1537}}& 1219 & {\multirow{4}{*}{1290}}\\[5ex]
    0.025 & 2630 & {\multirow{4}{*}{3053}} & 2478 & {\multirow{4}{*}{2556}} \\[5ex]
    0.0125 & 5201 & & 5031 & \\
    \bottomrule
  \end{tabular}
  \caption{Comparison of the computer running times (in seconds) of conforming and nonconforming time grids using Taylor-Hood and \textbf{P}2 elements with $h=1/32$.}
  \label{spcpu1hk}
\end{table}

\begin{table}[!htbp] 
  \centering
  \begin{tabular}{c ccccc}
  \toprule
   {\multirow{4}{*}{$\Delta t$}} & \multicolumn{2}{c}{Steklov-Poincar\'{e} Method}& {\multirow{4}{*}{$\Delta t$}}  & \multicolumn{2}{c}{Robin Method $(\alpha_f=1, \alpha_s=100$)} \\[2ex]
  \cmidrule{2-3}
  \cmidrule{5-6}
    & Conforming & Nonconforming & & Conforming & Nonconforming\\[2ex]
    \midrule
    0.4 & 2224 & {\multirow{4}{*}{2784}} & 0.8 & 5004 & {\multirow{4}{*}{5068}}  \\[5ex]
    0.2 & 4476 & {\multirow{4}{*}{5408}} & 0.4 & 9214 & {\multirow{4}{*}{9271}}\\[5ex]
    0.1 & 9468 & {\multirow{4}{*}{11137}}& 0.2 & 18390 & {\multirow{4}{*}{19673}}\\[5ex]
    0.05 & 18405 & {\multirow{4}{*}{19198}} & 0.1 & 36575 & {\multirow{4}{*}{37390}} \\[5ex]
    0.025 & 38525 & & 0.05 & 69892 & \\
    \bottomrule
  \end{tabular}
  \caption{Comparison of the computer running times (in seconds) of conforming and nonconforming time grids using MINI and \textbf{P}1 elements  with $h=1/64$.}
  \label{spcpu1hk2}
\end{table}

\begin{table}[!htbp] 
  \centering
  \begin{tabular}{ c || c ccccc }
    \hline
    $\alpha_s$ & 1 &3&5&10&50&100\\
    \hline
    Number of iterations & 31 &25&22&20&18&18 \\
    \hline
  \end{tabular}
  \caption{ Number of GMRES iterations for $T=0.2$ using $(\Delta t_f, \Delta t_s) = (0.025, 0.0125)$ and  $\alpha_f=1$.}
  \label{tabiter}
\end{table}

\subsection{Test 2 : Hemodynamic Experiment}
 In this example, we consider the blood flow problem reported in \cite{R14, R10}.  The domain and the boundary conditions used for the computation are depicted in Figure \ref{fig:fig12}. The force $\bb(t)$ applied to the left fluid boundary in Figure \ref{fig:fig12} denotes the stress at the inlet at $t$ seconds and is defined as:
\begin{equation*}
    \bb(t)= \begin{cases}
        \left(-10^3\left(1-\cos{\frac{2\pi t}{0.025}}\right),0\right) &  t \leq 0.025\\
        (0,0)) & 0.025<t<T.
    \end{cases}
\end{equation*}
The parameters used in this example are in accordance with the characteristics of blood flow in the human body.
The density of the fluid, $\rho_f$, is 1 g/$\mbox{cm}^3$ and the viscosity of the fluid, $\nu_f$, is 0.035 g/cm·s. The density of the structure, $\rho_s$, is
1.1 g/$\mbox{cm}^3$. The Young’s Modulus of the structure, $E$, is $3 \times 10^6$ dyne/$\mbox{cm}^2$ and the Poisson ratio, $\nu$, is 0.3. The Lame parameters $\lambda$ and $\nu_s$ are defined as follows:
\begin{equation*}
    \lambda = \frac{\nu E}{(1-2\nu)(1+\nu)} \mbox{ dyne/cm}^2, \quad \nu_s = \frac{E}{2(1+\nu)} \mbox{dyne/cm}^2.
\end{equation*}
Both the fluid and structure have volume forces of $\bf_s = \bf_f = \b0$ dyne/$\mbox{cm}^2$. Due to the closely matched densities between the fluid and the structure, the problem is significantly impacted by the added mass effect. Thus, when using a DD method at each time step, as in most DD approaches for FSI, additional relaxation steps may be necessary for solution stability, in addition to the use of a very fine time grid \cite{R10}.   

We simulate this example using the Steklov-Poincar\'{e} method
without encountering the  stability issue, as our local problems are still time-dependent. 
A uniform mesh is employed to spatially discretize the domains of both the fluid and structure, with $h_x$ and $h_y$ representing the spatial discretization in the $x$ and $y$ directions. For this test, the fluid and structure are approximated using MINI elements and \textbf{P}1 element, respectively. We use a time step of $\Delta t_f = 2 \times 10^{-4}$ for the fluid subdomain and  $\Delta t_s = 1\times 10^{-4}$ for the structure subdomain, with the final time set at $T = 0.1$. By setting $h_x=0.1$ cm and changing $h_y$ between 0.1 cm and $\frac{1}{30}$ cm, we monitor the vertical displacement at three distinct points on the interface (see Figure \ref{fig:my_label}). We observe similar vertical displacement at each point for all values of $h_y$. In \cite{R14} and \cite{R10}, it is observed that the solution heavily depends on spatial discretization, and the vertical displacements in Figure \ref{fig:my_label} are similar to their results obtained by fine spatial discretization (see Figure 9 of \cite{R10}). 
Furthermore, we report the interface velocity errors, 
$ \frac{1}{2}\|\bu - \dot \beeta\|^2_{\Gamma} $ for different mesh sizes of $h_y$  in Table \ref{taberrorinterface} at the final time $T=0.1$ seconds.
\begin{figure}[!htbp]
\begin{center}
\begin{tikzpicture}[scale=0.6]
\draw[thick, -] (0,6)--(0,0)--(11,0);

\draw[thick, -] (0,6)--(11,6);
\draw[thick, -] (0,4.5)--(11,4.5);
\draw[thick, -] (11,0)--(11,6);
\node[left] at (0,2) {$ \bu_{N} = \bb(t) $};
\node[left] at (0,5.2) {$ \beeta_{D} = 0 $};

\node[right] at (11,2) {$ \bu_{N} = 0 $};
\node[right] at (11,5.2) {$ \beeta_{D} = 0 $};

\node[below] at (5.5,0) {$ \bu_{D} = 0 $};
\node[above] at (5.5,6) {$ \beeta_{N} = 0 $};

\node[] at (5.5,5.2) {$ \Omega_s = [0,6] \times [1,1.1] $};
\node[] at (5.5,2) {$ \Omega_f = [0,6] \times [0,1] $};

\draw[thick, ->] (6.5,3.5)--(5.6,4.5);
\node[] at (6.8,3.5) {$ \Gamma $};

\end{tikzpicture}
\caption{Domain and boundary conditions for Test2}
\label{fig:fig12}
\end{center} \vspace{-0.4cm}
\end{figure}
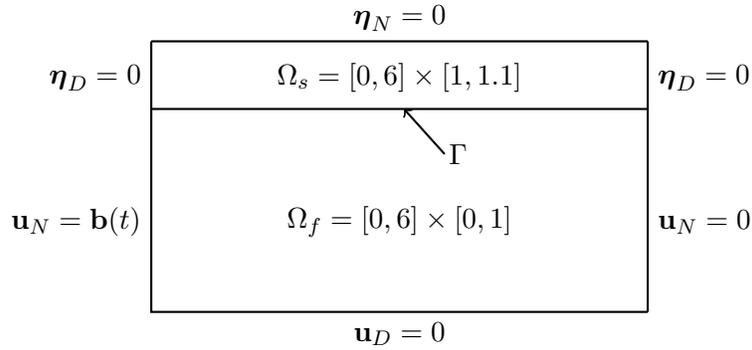

\begin{figure}[!htbp]
   \centering
  \scalebox{.32}{\includegraphics{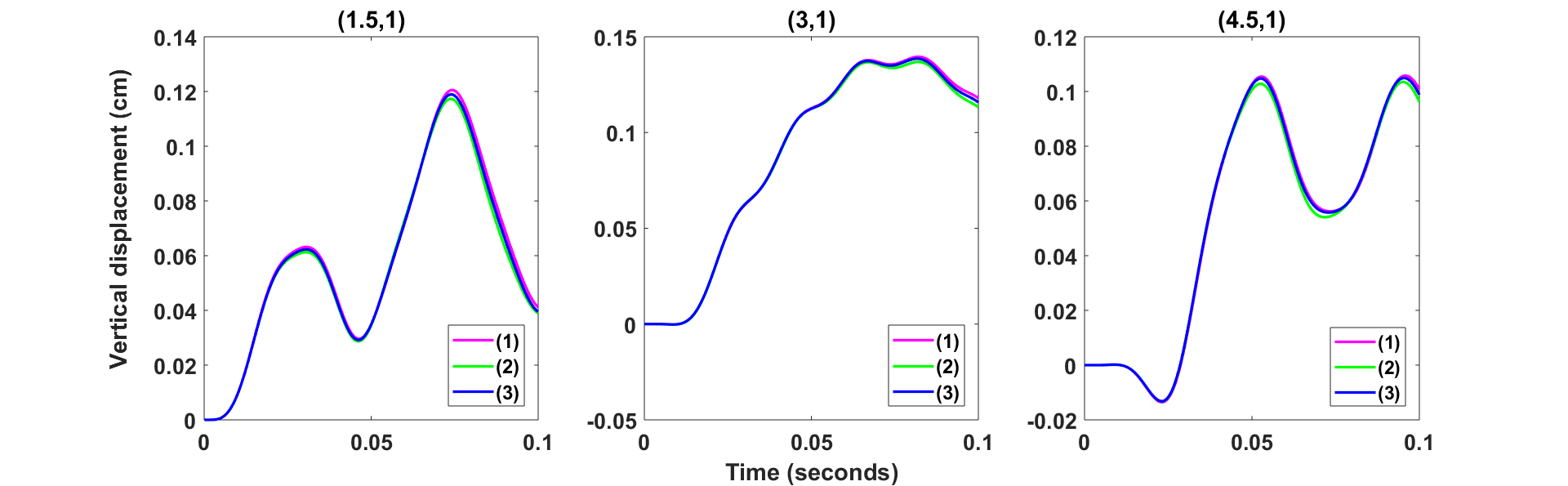}}
    \caption{Vertical displacement at three points on the interface with: (1) $h_x = 0.1$ cm, $h_y = 0.1$ cm, (2) $h_x = 0.1$ cm, $h_y = 0.05$ cm,
and (3) $h_x = 0.1$ cm, $h_y = \frac{1}{30}$ cm}
    \label{fig:my_label}
\end{figure}

\begin{table}[!htbp] 
  \centering
  \begin{tabular}{ c c}
    \toprule
    Values of $h_y$ & Interface Velocity Error \\
    \midrule
    0.1 & 2.05e-04\\[1.5ex]
    0.05 & 8.23e-05\\[1.5ex]
    $\displaystyle \frac{1}{30}$ & 2.87e-05\\[1.5ex]
    \bottomrule
  \end{tabular}
  \caption{Errors in the continuity of velocity between subsystems for $h_x=0.1$cm.}
  \label{taberrorinterface}
\end{table}

\section{Conclusion} \label{sec:conclusion}
 We have introduced global-in-time domain decomposition methods by formulating two interface problems, based on the Steklov-Poincar\'{e} operator and Robin transmission conditions, respectively. In these methods, the fluid and the structure subproblems are time-dependent and solved independently using local solvers. This allows for the use of nonconforming time grids and different time-stepping algorithms for local problems. The SWR algorithm for the second interface problem was introduced and analyzed in continuous and semi-discrete settings. In this paper, we emphasize the time discretization aspects, providing an in-depth analysis of the nonconforming time discretization method utilized in our numerical experiments. While a detailed exploration of spatial discretization techniques and stability analysis would further enhance our study, we propose it as a direction for future research. We performed numerical tests on two examples, including a non-physical problem where we tested with various mesh sizes and time steps to verify convergence rates. The use of non-conforming time grids resulted in better accuracy within similar computational times compared to results obtained using conforming coarse grids.
 For the physical benchmark problem, we implemented the Steklov-Poincaré method and observed similar results reported in the literature.
Local time stepping makes these algorithms efficiently applicable to multiphysics problems, where local problems may have different time scales. 
In our ongoing research, we are extending these algorithms to tackle more complicated multiphysics problems, such as the Stokes-Biot system.
\break


\begin{thebibliography}{50}
\bibitem{miniexp} D.N. Arnold, F. Brezzi, M. Fortin, A stable finite element for the Stokes equations, Calcolo 21 (1984) 337–344.

\bibitem{bfpt} F.P.T. Baaijens, A fictitious domain/mortar element method for fluid-structure interaction, Int. J. Numer. Methods Fluids 35 (2001) 743–61.

\bibitem{R11} S. Badia, A. Quaini, and A. Quarteroni, Splitting methods based on algebraic factorization for fluid-structure interaction, SIAM Journal on Scientific Computing 30 (2008) 1778–1805.

\bibitem{BRM17}F. Ballarin, G. Rozza, and Y. Maday, Reduced-order semi-implicit schemes for fluid-structure interaction problems. in Model Reduction of Parameterized Systems, P. Benner and M. Ohlberger and A. Patera and G. Rozza and K. Urban, Springer, (2017) 149-167. 

\bibitem{R1} Y. Bazilevs, V. M. Calo, T.J.R. Hughes, and Y. Zhang, Isogeometric fluid-structure interaction: theory, algorithms, and computations, Computational Mechanics 43 (2008) 3–37.

\bibitem{bdg} D. Boffi and L. Gastaldi, A fictitious domain approach with Lagrange multiplier for fluid-structure interactions, Numer. Math. 135 (2017) 711–732.

\bibitem{R2} J. Boujot, Mathematical formulation of fluid–structure interaction problems, Modelisation Math´ematique et Analyse Numerique 21 (1987) 239–260.


\bibitem {B} E. Burman, R. Durst, M,A. Fernández, and J. Guzm\'{a}n, Fully discrete loosely coupled Robin-Robin scheme for incompressible fluid–structure interaction: stability and error analysis, Numer. Math. 151 (2022) 807–840. 


\bibitem{A} E. Burman, R. Durst, and J. Guzm\'{a}n, Stability and error analysis of a splitting method using Robin–Robin coupling applied to a fluid-structure interaction problem, Numer. Methods Partial Differ. Eq. 38 (2022) 1396–1406.

\bibitem{BF09} E. Burman and M.A. Fern{\'a}ndez, Stabilization of explicit coupling in fluid-structure interaction involving fluid incompressibility, Computer Methods in Applied Mechanics and Engineering 198 (2009) 766-784.

\bibitem{R3} S. Canic, A. Mikelic, and J. Tambaca, A two-dimensional effective model describing fluid–structure interaction in blood flow: analysis, simulation and experimental validation, Comptes Rendus M\'{e}canique 333 (2005) 867–883.

\bibitem{R4} J. Chrispell and L. Fauci, Peristaltic pumping of solid particles immersed in a viscoelastic fluid, Mathematical Modelling of Natural Phenomena 6 (2011) 67–83.

\bibitem{Fer11} M. Fern{\'a}ndez, Coupling schemes for incompressible fluid-structure interaction: implicit, semi-implicit and explicit, 
SeMA Journal: Boletin de la Sociedad Espa\~{n}ola de Matem{'a}tica Aplicada 55 (2011) 59-108.

\bibitem{FGG05} M. Fern{\'a}ndez, J. Gerbeau, and C. Grandmont, A projection semi-implicit scheme for the coupling of an elastic structure with an incompressible fluid, INRIA, 2005 Technical Report RR-5700. 

\bibitem{gavin} K.J. Gavin, New subgrid artificial viscosity Galerkin methods for the Navier-Stokes equations, Comput. Methods Appl. Mech. Engrg. 200 (2011) 242-250.

\bibitem{phlee} P. G. Geredeli, H. Kunwar, and H. Lee, Partitioning method for the finite element approximation of a 3D fluid-2D plate interaction system, Numer. MethodsPartial Differ. Eq. (2024), e23132. 

\bibitem{gig} G. Gigante and C. Vergara, On the Choice of Interface Parameters in Robin–Robin Loosely Coupled Schemes for Fluid–Structure
Interaction, Fluids 6 (2021) 213.


\bibitem{GiraultRaviart} V. Girault and P.A. Raviart, Finite element methods for Navier-Stokes equations, Springer-Verlag (1986).


\bibitem{R5} N. Haritos, Introduction to the analysis and design of offshore structures – an overview, Electronic Journal of Structural Engineering 7 (2007) 55–65.

\bibitem{ff} F. Hecht, New development in freefem++, J. Numer. Math. 20 (2012) 251–265.

\bibitem{R12} M. Heil, An efficient solver for the fully coupled solution of large-displacement fluid–structure interaction problems, Comput. Methods Appl. Mech. Eng. 193 (2004) 1–23.

\bibitem{hab} C. Hesch, A.J. Gil, A. Carreno, J. Bonet, and P. Betsch, A mortar approach for Fluid-Structure interaction problems: Immersed strategies for deformable and rigid bodies, Comput. Methods Appl. Mech. Eng. 278 (2014) 853–82.

\bibitem{R6} J. J. Heys, T. A. Manteuffel, S. F. McCormick, and J.W. Ruge, First order system least squares (FOSLS) for coupled fluid–elastic problems, Journal of Computational Physics 195 (2004) 560–575.



\bibitem{HJJ13}
 T.T.P. Hoang, J. Jaffr\'e, C. Japhet, M. Kern, and J.E. Roberts, Space-time domain decomposition methods for diffusion problems in mixed formulations, SIAM J. Numer. Anal. 51 (2013) 3532-3559. 

\bibitem{HJK16}
T.T.P. Hoang, C. Japhet, M. Kern, and J.E. Roberts, Space-time domain decomposition for reduced fracture models in mixed formulation, SIAM J. Numer. Anal. 54 (2016) 288-316.

\bibitem{THH} T.T.P. Hoang, H. Kunwar, and H. Lee, Nonconforming time discretization based on Robin transmission conditions for the Stokes–Darcy system, Applied Mathematics and Computation. 413 (2022) 126602.

\bibitem{THe} T.T.P. Hoang and H. Lee, A Global-in-time Domain Decomposition Method for the Coupled Nonlinear Stokes and Darcy Flows, J. Sci. Comput. 87 (22)  (2021).

\bibitem{R13} J. Hron and S. Turek, A monolithic FEM/multigrid solver for an ALE formulation of fluid-structure interaction with applications in biomechanics, in Fluid-Structure Interaction, Springer (2006) 146-170.

\bibitem{Yotov23} M. Jayadharan, M. Kern, M. Vohralík, and I. Yotov, A space-time multiscale mortar mixed finite element method for parabolic equations, SIAM Journal on Numerical Analysis 61 (2023) 675–706.

\bibitem{R7} R. Kamakoti and W. Shyy, Fluid–structure interaction for aeroelastic applications, Progress in Aerospace Sciences 40 (2004) 535–558.

\bibitem{R14} P. Kuberry and H. Lee, A decoupling algorithm for fluid–structure interaction problems based on optimization, Comput. Methods Appl. Mech. Engrg. 267 (2013) 594-605.  

\bibitem{Klsee} H. Kunwar, H. Lee, K. Seelman, Second-order time discretization for a coupled quasi-Newtonian fluid-poroelastic system. Int J Numer Meth Fluids. 92 (2020) 687–702.

\bibitem{Lion} Pierre-Louis Lions, On the Schwarz Alternating Method III: A Variant for Nonoverlapping Subdomains, in Domain Decomposition Methods for Partial Differential Equations, T.F. Chan, R. Glowinski, J. Périaux, O.B. Widlund (Eds.) 1989.

\bibitem{vad} R.V. Loon, P.D. Anderson, J. de Hart, and F.P.T Baaijens, A combined fictitious domain/adaptive meshing method for fluid-structure interaction in heart valves, Int J. Numer. Methods Fluids 46 (2004) 533–44.

\bibitem{Martin} V. Martin, An optimized Schwarz waveform relaxation method for the unsteady convection diffusion equation in two dimensions. Appl. Numer. Math. 52 (2005) 401–428.

\bibitem{mtg} M. Mayr, T. Klöppel, W.A. Wall, and M.W. Gee, A temporal consistent monolithic approach to fluid–structure interaction enabling single field predictors, SIAM J. Sci. Comput. 37 (2015) 30–59.

\bibitem{mng} M. Mayr, M.H. Noll, and M.W. Gee, A hybrid interface preconditioner for monolithic fluid-structure interaction solvers, Adv. Model Simul. Eng. Sci. 7 (2020) 1-33.

\bibitem{R10} C.M. Murea and  S. Sy, A fast method for solving fluid-structure interaction problems numerically, International Journal for Numerical Methods in Fluids 60 (2009) 1149–1172.

\bibitem{R9} F. Nobile and C. Vergara, An effective fluid-structure interaction formulation for vascular dynamics by generalized Robin conditions, SIAM Journal on Scientific Computing 30 (2008) 731–763.

\bibitem{NBR21} M. Nonino, F. Ballarin, and G. Rozza, A Monolithic and a partitioned, reduced basis method for fluid–structure interaction problems, Fluids 6 (2021) 229-263.

\bibitem{JMP} J.P. Sheldon, S.T. Miller, and J.S. Pitt, A hybridizable discontinuous Galerkin method for modeling fluid–structure interaction, J. Comput. Phys. 326 (2016) 91–114.

\bibitem{R8} P. Le Tallec and S. Mani, Numerical analysis of a linearised fluid-structure interaction problem, Numerische Mathematik 87 (2000) 317–354.



\end{thebibliography}
\end{document}